\documentclass[a4paper]{amsart}

\usepackage{amsthm}
\usepackage{amssymb}
\usepackage{enumerate}
\usepackage{stmaryrd}
\usepackage{colonequals}
\usepackage[mathcal]{euscript}
\usepackage{tikz-cd} 
\usepackage[all]{xy}
\SelectTips{cm}{}

\usepackage{hyperref}
\hypersetup{%
  bookmarksnumbered=true,%
  colorlinks=true,%
  linkcolor=blue,%
  citecolor=blue,%
  filecolor=blue,%
  menucolor=blue,%
  urlcolor=blue,%
  bookmarksopen=true,%
  bookmarksdepth=2,%
  pageanchor=true}

\makeatletter
\@namedef{subjclassname@2020}{\textup{2020} Mathematics Subject Classification}
\makeatother

\numberwithin{equation}{section}

\swapnumbers

\theoremstyle{plain}
\newtheorem{theorem}[equation]{Theorem}
\newtheorem{proposition}[equation]{Proposition}
\newtheorem{lemma}[equation]{Lemma} 
\newtheorem{corollary}[equation]{Corollary}

\theoremstyle{definition}
\newtheorem{definition}[equation]{Definition}
\newtheorem{example}[equation]{Example}
\newtheorem{construction}[equation]{Construction}

\theoremstyle{remark}
\newtheorem{remark}[equation]{Remark} 

\newtheorem*{ack}{Acknowledgements}

\newcommand*{\intref}[2]{\def\tmp{#1}\ifx\tmp\empty\hyperref[#2]{\ref*{#2}}\else\hyperref[#2]{#1~\ref*{#2}}\fi}

\newcommand{\test}{A}

\newcommand{\cosupp}{\operatorname{cosupp}}
\renewcommand{\dim}{\operatorname{dim}}
\newcommand{\dcat}[1]{{\mathbf D}(#1)}
\newcommand{\dbcat}[1]{{\mathbf D}^{\mathrm b}(\mathrm{mod}\, #1)}

\newcommand{\Ext}{\operatorname{Ext}}

\newcommand{\fdim}{\operatorname{flat\,dim}}

\renewcommand{\ge}{\geqslant}

\newcommand{\GProj}{\operatorname{GProj}}
\newcommand{\uGProj}{\underline{\GProj}}
\newcommand{\uGproj}{\underline{\operatorname{Gproj}}}

\newcommand{\Hom}{\operatorname{Hom}}
\newcommand{\hysupp}{V^h}

\renewcommand{\Im}{\operatorname{Im}}

\newcommand{\Ker}{\operatorname{Ker}}

\renewcommand{\le}{\leqslant}

\renewcommand{\mod}{\operatorname{mod}}
\newcommand{\Mod}{\operatorname{Mod}}

\newcommand{\pos}[2][k]{#1\llbracket{#2}\rrbracket}

\newcommand{\Proj}{\operatorname{Proj}}
\newcommand{\rank}{\operatorname{rank}}

\newcommand{\RHom}{\operatorname{{\mathbf R}Hom}}

\newcommand{\StMod}{\operatorname{StMod}}
\newcommand{\stmod}{\operatorname{stmod}}

\newcommand{\supp}{\operatorname{supp}}
\newcommand{\Thick}{\operatorname{Thick}}
\newcommand{\Tor}{\operatorname{Tor}}

\newcommand{\col}{\colon}

\newcommand{\lotimes}{\otimes^{\mathbf L}}

\renewcommand{\setminus}{\smallsetminus}

\newcommand{\xra}{\xrightarrow}
\newcommand{\lra}{\longrightarrow}

\newcommand{\bikp}{Benson/Iyengar/Krause/Pevtsova}

\newcommand{\mcE}{E}

\newcommand{\mcV}{\mathcal{V}}

\newcommand{\bfC}{\mathbf C}

\newcommand{\bbA}{\mathbb A} 
 
\newcommand{\bbG}{\mathbb G}
\newcommand{\bbP}{\mathbb P}
 
\newcommand{\bbZ}{\mathbb Z} 

\newcommand{\bsa}{\boldsymbol{a}} 
\newcommand{\bsb}{\boldsymbol{b}} 
\newcommand{\bsf}{\boldsymbol{f}}

\newcommand{\bss}{\boldsymbol{s}} 
 
\newcommand{\bsx}{\boldsymbol{x}}

\newcommand{\fa}{\mathfrak{a}}

\newcommand{\fm}{\mathfrak{m}} 
\newcommand{\fn}{\mathfrak{n}} 
\newcommand{\fp}{\mathfrak{p}}
\newcommand{\fq}{\mathfrak{q}}

\newcommand{\si}{\sigma}

\newcommand{\pisupp}{\pi\text{-}\supp}
\newcommand{\picosupp}{\pi\text{-}\cosupp}

\newcommand{{\sgs}}{\mathrm{sgs}}
\newcommand{\ve}{\varepsilon}

\title[Rank varieties and $\pi$-points]
{Rank varieties and $\pi$-points for\\ elementary supergroup schemes}

\author[BIK$\Pi$]{Dave Benson, Srikanth B. Iyengar, \\ Henning Krause, and Julia Pevtsova}

\address{Dave Benson \\ 
Institute of Mathematics\\ 
University of Aberdeen\\ 
King's College\\ 
Aberdeen AB24 3UE\\ 
Scotland U.K.}

\address{Srikanth B. Iyengar\\ 
Department of Mathematics\\
University of Utah\\ 
Salt Lake City\\
UT 84112 \\
U.S.A.}

\address{Henning Krause\\ 
Fakult\"at f\"ur Mathematik\\ 
Universit\"at Bielefeld\\ 
33501 Bielefeld\\ 
Germany.}

\address{Julia Pevtsova\\
Department of Mathematics\\
University of Washington\\
Seattle, WA 98195\\
U.S.A.}

\begin{document}

\dedicatory{To Jon F. Carlson on his 80th birthday.}

\begin{abstract}
We develop a support theory for elementary supergroup schemes, over a field of positive characteristic $p\ge 3$, starting  with a definition of a $\pi$-point generalising cyclic shifted subgroups of  Carlson for elementary abelian groups and $\pi$-points of Friedlander and Pevtsova for finite group schemes. These are defined in terms of  maps from the graded algebra $k[t,\tau]/(t^p-\tau^2)$, where $t$ has even degree and $\tau$ has odd degree. The strength of the theory is demonstrated by classifying the parity change invariant localising subcategories  of the stable module category of an elementary supergroup scheme.
\end{abstract}



\date{August 6, 2020}

\maketitle

\setcounter{tocdepth}{1}
\tableofcontents

\section*{Introduction}

Carlson~\cite{Carlson:1983a}  introduced two notions of variety for a finitely generated module over an elementary abelian $p$-group. One, the rank variety, is based on restrictions to cyclic shifted subgroups, while the other is a cohomological support variety. This theory was generalised to infinitely generated modules by  Benson, Carlson and Rickard~\cite{Benson/Carlson/Rickard:1996a} by using cyclic shifted subgroups defined over extension fields where enough generic points exist. 

The notion of rank variety was put in the more general context of a finite group scheme $G$ over a field $k$  by Friedlander and Pevtsova \cite{Friedlander/Pevtsova:2007a}, through the theory of $\pi$-points. A $\pi$-point is a flat algebra homomorphism from $K[t]/(t^p)$ to $KG_K$, where $K$ is an extension field of $k$. There is an equivalence relation on $\pi$-points, and in the case of an elementary abelian group  there is exactly one shifted subgroup in each equivalence class over a large enough field. 

In a parallel development a theory of support varieties based on cohomology, and applicable in a rather broad context, was developed in \cite{Benson/Iyengar/Krause:2009a, Benson/Iyengar/Krause:2011a}. Combining those ideas with the theory of $\pi$-points eventually led to a classification of the localising, and colocalising, subcategories of the stable module category of a finite group scheme by the current authors~\cite{Benson/Iyengar/Krause/Pevtsova:2018a}.

In~\cite{Benson/Iyengar/Krause/Pevtsova:bikp5} we began a program to extend all these results to the world of supergroup schemes. In that work we identified a family of elementary supergroups schemes and proved that the projectivity of modules over a unipotent supergroup scheme can be detected by its restrictions to the elementary ones, possibly defined over extensions fields. This is in  analogy with Chouinard's theorem for finite groups. 

In this paper we develop a theory of $\pi$-points for the elementary supergroup schemes $\mcE$ introduced  in~\cite{Benson/Iyengar/Krause/Pevtsova:bikp5}, and classify the localising subcategories of its stable module category. This feeds into the proof of a similar classification for finite unipotent supergroup schemes, presented in \cite{Benson/Iyengar/Krause/Pevtsova:bikp9}. It transpires that  rather than flat maps from $K[t]/(t^p)$, we have to consider the $K$-algebra
\[
\test_K=K[t,\tau]/(\tau^2-t^p)
\]
where $t$ is in even degree and $\tau$ is in odd degree, and maps of finite flat dimension 
\[ 
\alpha\colon\test_K\to K\mcE_K \,.
\] 
The basic new result in our work is that $\pi$-points detect projectivity:  a $k\mcE$-module $M$ is projective if, and only if, for each $\pi$-point $\alpha$ as above the restriction of the $K\mcE_K$-module $K\otimes_kM$ to $\test_K$ has finite flat dimension; a corresponding statement involving $\Hom_k(K,M)$, rather than $K\otimes_kM$, also holds. 

From the point of view of commutative algebra, the group algebra $k\mcE$ is a complete intersection, and for such rings  Avramov~\cite{Avramov:1989a} has developed a theory of support sets for modules, as yet another extension of Carlson's work. Our proof of the detection theorem for $\mcE$ goes by relating $\pi$-points to support sets over $k\mcE$. However, we need a version of the theory that applies also to infinite dimensional modules. This is presented in the Appendix.

With the detection theorem on hand, we put an equivalence relation on
$\pi$-points, analogous to the one
in~\cite{Friedlander/Pevtsova:2007a}, and exhibit an explicit set of
representatives of these equivalence classes, up to linear multiples,
analogous to the cyclic shifted subgroup approach in
\cite{Benson/Carlson/Rickard:1996a,Carlson:1983a}. These form a
projective space over $k$, and give rise to various notions of support
for modules over $k\mcE$. This allows us to classify the localising
subcategories of  the stable category of $k\mcE$-modules. For a finite
dimensional $k\mcE$-module $M$ one also has a rank variety, in exact analogy with Carlson's theory of rank varieties for elementary abelian groups. This variety is determined by a rank condition on an explicitly defined matrix, and in particular it is a closed subset.

\subsection*{Outline of the paper}
The basic definitions concerning supergroup schemes, including the structure of the elementary ones, is recalled in \intref{Section}{se:prelim}.
In \intref{Section}{se:supportsets} we record the desired statements concerning support sets of modules over $k\mcE$, by specialising results for general complete intersections established in \intref{Appendix}{se:appendix}. The link to $\pi$-points, and a proof of the detection theorem discussed above, is presented in \intref{Section}{se:pi-points}. The proofs again require quite substantial input from the homological theory of complete intersection rings, and also  basic facts about the representation theory of the algebra $\test$ recalled  in \intref{Section}{se:test}.  From this point on, the narrative unfolds in the expected way: Rank varieties for finite dimensional modules are introduced in \intref{Section}{se:rank}, leading to an explicit method for computing them. The equivalence relation on $\pi$-points is discussed in \intref{Section}{se:pi-and-cohomology}. \intref{Section}{se:support-and-cosupport} brings in the cohomological notions of support and cosupport, culminating with the classification results.

\begin{ack}
  It is a pleasure to acknowledge the support provided by the American
  Institute of Mathematics in San Jose, California, through their
  ``Research in Squares" program.  We also acknowledge the National Science
  Foundation under Grant No.\ DMS-1440140 which supported three of the
  authors (DB, SBI, JP) while they were in residence at the Mathematical Sciences
  Research Institute in Berkeley, California, during the Spring 2018
  semester.  Finally, two of the authors (DB, JP) are grateful for hospitality
  provided by City, University of London.

  SBI was partly supported by NSF grants DMS-1700985 and DMS-2001368.
  JP was partly supported by NSF grants DMS-1501146,
  DMS-1901854, and a Brian and Tiffinie Pang faculty fellowship.
\end{ack}

\section{Elementary supergroup schemes}
\label{se:prelim}
Throughout $k$ will be a field of positive characteristic $p\ge 3$. A \emph{superalgebra} will mean a $\bbZ/2$-graded algebra, and a \emph{graded} module over such an algebra will be assumed to be a $\bbZ/2$-graded left module. The category of graded modules over a superalgebra $A$ is denoted $\Mod(A)$; these are allowed to be infinitely generated. The full subcategory of finitely generated ones is denoted $\mod(A)$. When $M$ is a graded $A$-module, $\Pi M$ denotes the module with the zero and one components swapped. For $m\in M$, the action of an element $a$ on $\Pi m$ is given by 
\[
a\cdot \Pi m \colonequals (-1)^{|a|}\Pi(am) \quad\text{in $\Pi M$.}
\]
A superalgebra $A$ is \emph{commutative} if  $yx=(-1)^{|x||y|}xy$ for all $x,y$ in $A$.

An \emph{affine supergroup scheme} over $k$ is a covariant  functor from  commutative superalgebras   to groups, whose underlying functor to sets is representable.  If $G$ is a supergroup scheme  its  \emph{coordinate ring} $k[G]$ is the representing object. By applying Yoneda's lemma to the group multiplication and inverse maps, it is a commutative Hopf superalgebra. This gives a contravariant equivalence of categories between affine supergroup schemes and commutative Hopf superalgebras.

A \emph{finite} supergroup scheme $G$ is an affine supergroup scheme
whose coordinate ring is finite dimensional. In this case, the dual
$kG=\Hom_k(k[G],k)$ is a finite dimensional cocommutative Hopf
superalgebra called the \emph{group ring} of $G$. This gives a
covariant equivalence of categories between finite supergroup schemes
and finite dimensional cocommutative Hopf superalgebras. We denote by
$\StMod(kG)$ the \emph{stable module category} which is obtained from
$\Mod(kG)$ by annihilating all projective modules. Note that this
carries the structure of a triangulated category since $kG$ is a
self-injective algebra. We write $\stmod(kG)$ for the full subcategory
of finite dimensional modules.

Following~\cite{Benson/Iyengar/Krause/Pevtsova:bikp5} we say that a finite supergroup scheme $E$ over $k$ is \emph{elementary} if it is isomorphic to a quotient of $E_{m,n}^-\times(\bbZ/p)^r$, where $E_{m,n}^-$ is the Witt elementary, explicitly described in~\cite[Definition~3.3]{Benson/Iyengar/Krause/Pevtsova:bikp5}. The main theorem of that paper states that if $G$ is a unipotent finite supergroup scheme over $k$ then a $kG$-module $M$ is projective  if and only if, for all extension fields $K$, and all elementary sub-supergroup schemes $\mcE$ of $G_K$, the module $M_K=K\otimes_kM$ is a projective $K\mcE$-module. It also gives a similar condition for the nilpotence of an element of cohomology.

We will  be concerned \emph{only} with the algebra structure of $kE$
and the existence of a comultiplicaton; the explicit formula for the latter plays no role. Nevertheless, here is a brief description of the supergroup schemes $E_{m,n}^-$ and their finite quotients.

Let $\bbG_{a(n)}$ be the $n$th Frobenius kernel of the additive group scheme $\bbG_a$. We denote by $\bbG_a^-$ the supergroup scheme with the group algebra $\Lambda^*(\sigma) \cong k[\sigma]/\sigma^2$ with generator $\sigma$ in odd degree and primitive. Let $W_m$ be the affine group scheme of Witt vectors of length $m$, and let $W_{m,n}$ be the $n$th Frobenius kernel of $W_m$. Hence, $W_{m,n}$ is a finite connected group scheme of height $n$. Then the finite group scheme $E_{m,n}$ is defined as a quotient 
\[
E_{m,n} : = W_{m,n}/W_{m-1,n-1}\,.
\]  
The \emph{Witt elementary} super group $E_{m,n}^-$ is a finite supergroup scheme determined uniquely by the following two properties: 
\begin{enumerate} 
\item It fits into an extension $1 \to E_{m,n} \to E_{m,n}^- \to \bbG_a^- \to 1$,
\item $kE_{m,n}^- \cong k[s_1,\dots,s_{n-1},s_n,\si]/(s_1^p, \dots,s_{n-1}^p,s_n^{p^m},\sigma^2-s_n^{p})$. 
\end{enumerate} 
See \cite[\S 8]{Benson/Iyengar/Krause/Pevtsova:bikp5} for details.

The quotients of $E_{m,n}^-\times(\bbZ/p)^r$ can be completely classified using the theory of Dieudonn\'e modules and, up to isomorphism, fall into one of the following  classes:
\begin{enumerate}[{\quad\rm(i)}]
\item
$\bbG_{a(n)} \times (\bbZ/p)^{\times s}$ with $n,s\ge 0$,
\item
$\bbG_{a(n)} \times \bbG_a^- \times (\bbZ/p)^{\times s}$ with $n,s\ge 0$,
\item
$E_{m,n}^- \times (\bbZ/p)^{\times s}$ with $m\ge 1$, $n\ge 2$, $s\ge 0$, or
\item
$E_{m,n,\mu}^-\times (\bbZ/p)^{\times s}$ with $m,n\ge 1$, $0\ne\mu\in k$ and $s\ge 0$.
\end{enumerate}
The last family involves supergroups $E^-_{m,n,\mu}$ which are
quotients of $E_{m+1, n+1}^-$ but have the same group algebra structure
as $E_{m+1, n}^-$.  In the present work we are only concerned with the
algebra structure of $kE$, so we do not distinguish between cases
(iii) and (iv). Therefore, ignoring the comultiplication, the group
algebra of interest is one of the following:
\[
 kE \cong 
 \begin{cases} k[s_1,\dots,s_n]/(s_1^p,\dots,s_n^p) \\
k[s_1,\dots,s_n,\sigma]/(s_1^p,\dots,s_{n}^p, \si^2) \\
k[s_1,\dots,s_n,\sigma]/(s_1^p,\dots,s_{n-1}^p,s_n^{p^{m}},s_n^p-\si^2)  \quad \text{with $m\ge 2$.}
\end{cases} 
\]
where the $s_i$ have  degree $0$ and $\sigma$ has  degree $1$.  The first case occurs when the supergroup scheme is a group scheme. For these the representation theory has been analysed in detail in \cite{Benson/Carlson/Rickard:1996a, Benson/Iyengar/Krause:2011b,  Carlson:1983a,  Friedlander/Pevtsova:2007a}. 

Our goal is to write down suitable analogues of the main theorems of
these papers for elementary supergroup schemes  of the second and
third form. Among these, the third case is the one that presents and
most challenges and is the focus of the bulk of this work. The second case is discussed in the last \intref{Section}{se:others}.

Henceforth $k$ will be, as before, a field of positive characteristic $p\ge 3$ and $\mcE$ the supergroup scheme with group algebra
\begin{equation}
\label{eq:kE}
k\mcE\colonequals \frac{k[s_1,\dots,s_n, \sigma]}{(s_1^p, \dots,s_{n-1}^p, s_n^{p^m}, s_n^p-\sigma^2)}
\end{equation}
with $|s_i|=0$ and $|\sigma|=1$.

\section{Support sets}
\label{se:supportsets}
In this section we describe supports sets of  $k\mcE$-modules, following the general theory for complete intersections developed in the Appendix. 

We write $\bbA^n(k)$ for the $n$-tuples of elements, $
(a_1,\dots,a_n)$, of $k$, and $\bbP^{n-1}(k)$ for the non-zero
elements of $\bbA^n(k)$ modulo scalar multiplication. The image of
$(a_1,\dots,a_n)$ in $\bbP^{n-1}(k)$ is denoted $[a_1,\dots,a_n]$.

For a singly or doubly graded ring $R$ we write $\Proj R$ for the set
of homogeneous prime ideals other than the maximal ideal of non-zero
degree elements, topologised with the Zariski topology.

\begin{construction}
\label{con:powerseries}
Consider the  ring of formal power series
\[
P \colonequals \pos{\bss, \sigma} \colonequals \pos{s_1,\dots,s_n, \sigma}
\]
in  indeterminates $\bss, \sigma$. This is a noetherian local ring and $\fm\colonequals (\bss,\sigma)$ is its maximal ideal. In the ring $P$ consider the ideal
\begin{equation}
\label{eq:mI}
I=(s_1^p,\dots,s_{n-1}^{p},  s_n^{p^m}, s_n^p - \sigma^2)\,.
\end{equation}
As $k$-algebras one has  $k\mcE = P/I$.  Moreover $s_1^p,\dots,s_{n-1}^{p},  s_n^{p^m}, s_n^p - \sigma^2$ is a regular sequence in $P$.  This can be verified directly from the definition, recalled in \intref{Appendix}{se:appendix}. Alternatively, one observes that this is a sequence of length $n+1$ in the local ring $P$ that is regular of Krull dimension $n+1$, and that the quotient ring, namely $k\mcE$, has Krull dimension zero; see \cite[Theorem~2.1.2]{Bruns/Herzog:1998a}. 
 
The upshot is that $k\mcE$ is complete intersection. In particular the constructions and results presented in \intref{Appendix}{se:appendix} apply to $k\mcE$.  One remark is in order before we can march on.
\end{construction}

\begin{remark}
\label{re:projectives-ungraded}
Recall that $k\mcE$ is a superalgebra. A  graded $k\mcE$-module $M$ is projective if and only if it is projective when viewed as an (ungraded) $k\mcE$-module.

One way to verify this is to note that $k\mcE$ is an artinian local ring, and hence projective modules (graded or not) are free. It is clear that being free does not  depend on the grading.
\end{remark}

\begin{definition}
Fix  $\bsa \colonequals [a_1,\dots,a_{n+1}]$ in $\bbP^{n}(k)$ and consider the hypersurface  $P_{\bsa}\colonequals P/(h_{\bsa})$, defined by the polynomial
\begin{equation}
\label{eq:hypersurface}
h_{\bsa} (\bss,\sigma)\colonequals a_1s_1^p + \cdots a_{n-1}s_{n-1}^p + a_ns_n^{p^m} + a_{n+1}(s_n^p-\sigma^2)\,.
\end{equation}
Since $h_{\bsa}(\bss,\sigma)$ is in $I$ there is a surjection  $\beta_{\bsa}\colon P_{\bsa}  \to k\mcE$. Given any $k\mcE$-module $M$  write $\beta^*_{\bsa}(M)$ for its restriction to $P_{\bsa}$, and set
\begin{equation}
\label{eq:Av-variety}
\hysupp_\mcE(M)\colonequals \{\bsa\in \bbP^{n}(k) \mid \fdim \beta_{\bsa}^*(M) =\infty \}\,.
\end{equation}
This is the \emph{support set} of $M$ defined through hypersurfaces, following Avramov~\cite{Avramov:1989a}.
\end{definition}

\begin{definition}
\label{de:base-change}
Given a $k\mcE$-module $M$ and an extension field $K$ of $k$ set
\[
M_K \colonequals K\otimes_k M \quad\text{and}\quad M^K \colonequals \Hom_k(K,M)
\]
viewed as $K\mcE$-modules in the natural way. The tensor product and Hom are taken in the graded category,  with $K$ in degree $0$. When  $K$ is a finite extension of $k$,  the $K\mcE$-modules $M_K$ and $M^K$ are isomorphic.

Observe that $M_K$ and $M^K$ can also be realised as
\[
M_K = K\mcE\otimes_{k\mcE} M \quad\text{and} \quad M^K = \Hom_{k\mcE}(K\mcE,M)\,.
\]
This reconciles the notation here with that introduced in \intref{Definition}{de:extending-fields}.
\end{definition}

\begin{theorem}
\label{th:support-sets}
Let $M$ be a $k\mcE$-module. The following conditions are equivalent:
\begin{enumerate}[\quad\rm(1)]
\item
$M$ is projective;
\item
$\hysupp_{\mcE_K}(M_K)=\varnothing$ for any field extension $K$ of $k$;
\item
$\hysupp_{\mcE_K}(M^K)=\varnothing$ for any field extension $K$ of $k$.
\end{enumerate}
In \emph{(2)} and \emph{(3)} it suffices to take  $K=k$ if $\rank_kM$ is finite and $k$ is algebraically closed. In any case it suffices to take for $K$ an algebraically closed extension of $k$ of transcendence degree at least $n$.
\end{theorem}

\begin{proof}
Observe that $M$ is projective if and only if it has finite projective dimension; this holds, for example, because  $k\mcE$ is finite dimensional and commutative.  Thus putting together \intref{Theorem}{th:Av} and \intref{Theorem}{th:appendix} yields the desired result.
\end{proof}

\subsection*{Cohomology of $k\mcE$}
The cohomology ring of $k\mcE$ was recorded in  \cite[Theorems~5.10 and~5.11]{Benson/Iyengar/Krause/Pevtsova:bikp5}. It takes the form
\begin{equation}
\label{eq:Ecohomology}
H^{*,*}(\mcE,k)= \Lambda(u_1,\dots,u_n)   \otimes_k k[x_1,\dots,x_n,\zeta]
\end{equation}
where the degrees are $|u_i|=(1,0)$, and $|x_i|=(2,0)$, $|\zeta|=(1,1)$. Here, the first degree is cohomological and the second comes from the $\bbZ/2$-grading on $k\mcE$. The numbering is chosen so that $x_i$ and $u_i$ restrict to zero on the subalgebra generated by $s_j$ if $i\ne j$, and $\zeta$ has non-zero restriction to the subalgebra generated by $\sigma$.

The nil radical of $H^{*,*}(\mcE,k)$ is generated by $ u_1,\dots,u_n$. The quotient modulo this ideal is
\[
H^{*,*}(\mcE,k)/\sqrt{0} = k[\bsx,\zeta] \quad\text{where}\quad  \bsx=x_1,\dots,x_n\,.
\]
This ring  has the property that the internal degree of a non-zero element is congruent to its cohomological degree modulo two.

It follows that if we take the $\Proj$ of  this ring as a doubly graded ring or as a singly graded ring by ignoring the internal degree, we get the same homogeneous prime ideals, that is to say, the natural map is a homeomorphism:
\[
\Proj H^{*,*}(E,k) \xra{\ \cong\ } \Proj H^{*}(E,k)\,.
\]
 Moreover the inclusion $k[\bsx,\zeta^2]\subset k[\bsx,\zeta]$ induces an isomorphism
 \[
\Proj k[\bsx,\zeta]\xra{\ \cong \ } \Proj k[\bsx,\zeta^2]\,.
 \]
 Thus $\Proj H^{*,*}(\mcE,k)$ identifies with the standard projective space. 
 
\subsection*{Cohomology of $P_{\bsa}$}
Fix $\bsa=[a_1,\dots,a_{n+1}] \in \bbP^n(k)$.  The cohomology of the hypersurface $P_{\bsa}$ is
\begin{equation}
\label{eq:Pacohomology}
\Ext^*_{P_{\bsa}}(k,k) =
\begin{cases}
\Lambda(v_1,\dots,v_n)\otimes_k k[\eta]  & \text{if $a_{n+1}\ne 0$} \\
\Lambda(v_1,\dots,v_n)\otimes_k k[y] \otimes \Lambda(\eta) & \text{if $a_{n+1}=0$}
\end{cases}
\end{equation}
where the $v_i$ and $\eta$ have cohomological degree one, and $y$ is of degree two; see, for example, ~\cite[Theorem~5]{Sjodin:1976a}. The map $\beta_{\bsa}\colon P_{\bsa}\to k\mcE$ induces a map of $k$-algebras
\[
\beta^*_{\bsa}\colonequals \Ext^*_{\beta_{\bsa}}(k,k)\colon \Ext_{\mcE}^{*}(k,k)  \lra \Ext^{*}_{P_{\bsa}}(k,k)\,.
\]
In the notation of \eqref{eq:Ecohomology} and \eqref{eq:Pacohomology}, and for $1\le i\le n$, this map is given by
\begin{align*}
&\beta^*_{\bsa}(u_i) = v_i  \\
&\beta^*_{\bsa}(x_i) = 
\begin{cases}
\displaystyle{\frac{a_i}{a_{n+1}}} \eta^2 & \text{when $a_{n+1}\ne 0$} \\
a_i y & \text{when $a_{n+1}=0$} 
\end{cases} \\
&\beta^*_{\bsa}(\zeta) = \eta  
\end{align*}
Let $\fp(\bsa)$ denote the radical of $\Ker(\beta^*_{\bsa})$; it is a prime ideal because the target  modulo its radical is a domain. A direct computation reveals that
\begin{equation}
\label{eq:beta-prime}
\fp(\bsa) = 
\begin{cases}
\left(a_{n+1}x_i - a_i \zeta^2\mid 1\le i \le n\right) +(u_1,\dots,u_n) & \text{if $a_{n+1}\ne 0$}\\
(\zeta)+ \left(a_ix_j - a_j x_i\mid 1\le i < j\le n\right)  + (u_1,\dots,u_n) & \text{if $a_{n+1}= 0$.}
\end{cases}
\end{equation}
So one gets a map
\[
\bbP^{n}(k) \lra \Proj H^*(\mcE,k) \quad\text{where}\quad \bsa \mapsto \sqrt{\Ker(\beta^*_{\bsa})}\,.
\]
When $k$ is algebraically closed, this map is an isomorphism onto the closed points, and through this we can identify $\hysupp_{\mcE}(M)$ with a  subset of $\Proj H^*(\mcE,k)$.

\section{The algebra $\test$}
\label{se:test}
 In this section we discuss modules over the superalgebra
\begin{equation}
\label{eq:fT}
\test \colonequals \frac{k[t,\tau]}{(t^p-\tau^2)}
\end{equation}
where $t$ is of even degree and $\tau$ is of odd degree. This algebra is commutative but not graded-commutative  since the odd degree element $\tau$ does not square to zero.  As a ring $\test$ is a noetherian domain of Krull dimension one.

We are interested in the representation theory of $\test$, and this is governed by the maximal Cohen--Macaulay modules. We will be interested also in the infinitely generated analogues of the MCM  modules, namely the \emph{G-projective modules}, where ``G" stands for ``Gorenstein". In what follows, we speak of ``G-projectives" rather than MCM modules, for consistency.  

\subsection*{G-projective modules}
Recall that the ring $\test$ is Gorenstein of Krull dimension one.  A module $M$ is G-projective if and only if $\Ext_{\test}^1(M,\test)=0$; see \cite[\S4]{Christensen:2000a}.  for details, including the definition of G-projectives. In particular,  the first syzygy module of any  $\test$-module is G-projective. Moreover, the finitely generated G-projective modules are precisely the MCM; \cite[\S2.1]{Bruns/Herzog:1998a} and also \cite[Theorems~3.3.7 and 3.3.10(d)]{Bruns/Herzog:1998a}.  Since there are only finitely many (isomorphism classes of) indecomposable finitely generated $\test$-modules, each G-projective module is a direct sum of finitely generated modules; see Beligiannis~\cite[Theorem~4.20]{Beligiannis:2011a}. 

Let $\GProj(\test)$ be the full subcategory of $\Mod(\test)$ consisting of  G-projective modules. It is a Frobenius exact category, with projectives (and hence injectives) the  projective $\test$-modules. We write $\uGProj(\test)$ for its stable category, viewed as a triangulated category.  One has an approximation functor $\Mod(\test)\to\uGProj(\test)$, which induces a right adjoint to the inclusion $\uGProj(\test)\subseteq \underline{\Mod}(\test)$ into the category of all $A$-modules modulo projectives. Thus  an $\test$-module has finite projective dimension if and only if its image in $\uGProj(\test)$ under the approximation functor is zero; this is why we care about the latter category. We write $\uGproj(\test)$ for the subcategory $\uGProj(\test)$ consisting of finitely generated modules.

Ignoring  grading, the classification of the indecomposable finitely generated G-projective $\test$-modules appears in Jacobinski~\cite{Jacobinski:1967a}, and is incorporated in the ADE classification, for $\test$ is an $A_{p-1}$ singularity; see   \cite[Proposition~5.11]{Yoshino:1990a}.  Besides $A$ itself,  these are the ideals $M_i\colonequals (t^{i},\tau)$ for $1\le i\le (p-1)/2$; observe that $M_{p-i}\cong M_i$ as ungraded modules. 

Taking grading into account, the $\{M_{i}\}_{i=1}^{p-1}$ are all the isomorphism classes of graded, non-projective, G-projective modules, up to shift. Thus $M_i$ has generators $\alpha$ and $\beta$ of degrees zero and one, respectively, and satisfy $\tau\alpha=t^i\beta$, $t^{p-i}\alpha=\tau\beta$. The module $\Omega(M_i)$ looks  similar, but with $\alpha$ in degree one and $\beta$ in degree zero, so $\Omega(M_i)\cong M_{p-i}$. The minimal  free resolution of $M_i$ is the complex
\begin{equation}
\label{eq:resol}
\cdots \to 
\test \oplus \Pi\test \xrightarrow{\left(\begin{smallmatrix} \tau & t^i \\ -t^{p-i} &
    -\tau\end{smallmatrix}\right)} \Pi\test\oplus \test \xrightarrow{\left(\begin{smallmatrix} \tau & t^i \\
    -t^{p-i} & -\tau \end{smallmatrix}\right)} \test\oplus \Pi\test \to 0\,. 
    \end{equation}
Here $\Pi\test$ is a free $\test$-module on a single odd degree generator.   

The Auslander--Reiten quiver for ungraded finitely generated G-projective   $\test$-modules is
\[ 
\xymatrix{[\test] \ar@<.7ex>[r] & [M_1] \ar@<.7ex>[l]\ar@<.7ex>[r] & [M_2]
  \ar@<.7ex>[l] \ar@<.7ex>[r] & \cdots \ar@<.7ex>[l] \ar@<.7ex>[r]&
  *[]{\, [M_{\frac{(p-1)}{2}}]\qquad\quad }\ar@(ur,dr)[] \ar@<.7ex>[l]  } 
  \]
See \cite[(5.12)]{Yoshino:1990a}. For graded modules, this unfolds  to give
\[ 
\xymatrix@=6mm
{[\test]\ar@<.7ex>[r] & [M_1] \ar@<.7ex>[l] \ar@<.7ex>[r] & \cdots \ar@<.7ex>[l] \ar@<.7ex>[r] &
  [M_{\frac{(p-1)}{2}}] \ar@<.7ex>[l]\ar@<.7ex>[r] & [M_{\frac{(p+1)}{2}}]
  \ar@<.7ex>[l] \ar@<.7ex>[r] & \cdots \ar@<.7ex>[l] \ar@<.7ex>[r] &
  [M_{p-1}] \ar@<.7ex>[l] \ar@<.7ex>[r] & [\Pi\test] \ar@<.7ex>[l]} 
\]
Ignoring the projectives, this is the same as the Auslander--Reiten quiver for ungraded $k[t]/(t^p)$-modules; cf.~ \intref{Proposition}{pr:detect}.

\subsection*{Tensor structures}
The $k$-algebra $\test$ has a canonical structure of a cocommutative  Hopf superalgebra over $k$, where the  elements  $\tau$ and $t$ are primitives.  Its counit, or augmentation, is the map of $k$-algebras 
\[
\test\lra k\quad\text{where $t,\tau\mapsto 0 $.}
\]
The tensor product of any $\test$-module with a projective is projective, and hence the same is true of modules of finite projective dimension. Thus  $\uGProj(\test)$ is a tensor triangulated category and the approximation functor is a tensor functor. 

The ideal $(\tau)$ in $\test$ is primitively generated, so the quotient $k[t]/(t^p)$ inherits a Hopf structure compatible with the surjection $\test\to k[t]/(t^p)$. So $\StMod(k[t]/(t^p))$, the stable category of ungraded $k[t]/(t^p)$-modules is also tensor triangulated, and restriction followed by the approximation functor induces a tensor  functor
\begin{equation}
\label{eq:stable-TT}
\StMod(k[t]/(t^p)) \lra \uGProj(\test)\,.
\end{equation}
The result below is well-known and easy to verify.

\begin{proposition}
\label{pr:detect}
The functor \eqref{eq:stable-TT} is an equivalence of categories; it restricts to an equivalence $\stmod(k[t]/(t^p))\simeq \uGproj(\test)$. \qed
\end{proposition}

Next we present a criterion for detecting when certain $A$-modules
have finite flat dimension. This is used in \intref{Section}{se:rank}
to construct rank varieties for modules over elementary supergroup
schemes. It is also with this application in mind that we switch to
speaking of flat dimension of a module, rather than its projective
dimension. Observe that one is finite if and only if the other is
finite, for $\dim A=1$.

\begin{definition}
\label{de:maximal-rank}
Any $k$-linear map $f\colon V\to V$ of vector spaces with $f^2=0$ satisfies $\Im(f)\subseteq \Ker(f)$; when equality holds we say $f$ has \emph{maximal image}. When $\rank_kM$ is finite the preceding inclusion yields inequalities
\[
\rank_k f \colonequals \rank_k\Im(f) \le \frac{\rank_k V}2 \le \rank_k\Ker(f)\,.
\]
Evidently $f$ has maximal image precisely when $\rank_kf =  (\rank_kV)/2$. The result below is a super analogue of \cite{Avramov/Iyengar:2018a}. \end{definition}

\begin{theorem}
\label{th:rank}
Let $M$ be a graded $\test$-module on which $t$ (equivalently, $\tau$) is nilpotent. Then $M$ has finite flat dimension
if and only if the square zero map
\begin{equation}
\label{eq:2x2}
 \begin{pmatrix} \tau & t \\ -t^{p-1} & -\tau \end{pmatrix} \colon M\oplus M \lra M\oplus M
\end{equation}
has maximal image. If $\rank_kM$ is finite, this happens if and only if the rank of the map above is equal to $\rank_kM$; otherwise the rank is strictly less.
\end{theorem}

\begin{proof}
 Given that $\tau^2=t^p$ in $\test$, when $t$ or $\tau$ is nilpotent on $M$, both are nilpotent and then $M$ is $(\tau,t)$-torsion.  Observe that  $M$ has finite flat dimension if and only if it has finite injective dimension, and since $\dim \test=1$, the latter condition is equivalent to the injective dimension of $M$ being at most one; this follows, for example, from \intref{Lemma}{le:pd-test}, for $A$ is a hypersurface.
 
Since $M$ is $(\tau,t)$-torsion, the minimal injective resolution of
$M$ consists only of copies of the injective hull of $k$. It follows that $M$ has injective dimension at most one if and only if  $\Ext^2_{\test}(k,M)=0$. Using the resolution \eqref{eq:resol} with $i=1$ for $J_1=k$, we see that $\Ext^2_{\test}(k,M)$ is isomorphic to the middle homology in the sequence
\[ 
\Pi M \oplus M
  \xrightarrow{\left(\begin{smallmatrix}\tau & t^i \\ -t^{p-i} & -\tau \end{smallmatrix}\right)}
M \oplus \Pi M 
  \xrightarrow{\left(\begin{smallmatrix}\tau & t^i \\ -t^{p-i} & -\tau \end{smallmatrix}\right)}
\Pi M \oplus M. 
\]
This justifies the first claim in the statement of the theorem. It also follows that when $\rank_kM$ is finite,  $\Ext^2_{\test}(k,M)=0$  if and only if the rank of the matrix \eqref{eq:2x2} on $M \oplus M$ is equal to $\rank_kM$, and otherwise the rank is strictly less. 
\end{proof}

\subsection*{Cohomology}
The ring $\test$ is a hypersurface and its cohomology ring
\[
  H^{*,*}(\test,k)\colonequals \Ext^{*,*}_{\test}(k,k)
\]
is well-known---see \cite[Theorem~5]{Sjodin:1976a}---and easy to compute using the minimal resolution of $k$ given in \eqref{eq:resol}, for $i=1$. Namely, there is an isomorphism of superalgebras
 \begin{equation}
 \label{eq:Acohomology}
 H^{*,*}(\test,k) \cong \Lambda(u)   \otimes_k k[\eta]
 \end{equation}
with $|u|=(1,0)$  and $|\eta|=(1,1)$.

\section{$\pi$-points}
\label{se:pi-points}
As before $k$ is a field of positive characteristic $p\ge 3$. We refocus on elementary supergroup schemes $\mcE$ over $k$ whose group algebra is of the form \eqref{eq:kE}.  In this section we introduce $\pi$-points for $\mcE$, based on maps from the $k$-algebra $\test$ discussed in \intref{Section}{se:test}. The main result is \intref{Theorem}{th:factorisation} that relates support defined via $\pi$-points to support sets introduced in \intref{Section}{se:supportsets}. As a consequence, we deduce that $\pi$-points detect projectivity; see \intref{Theorem}{th:detection}.

\begin{definition}
\label{defn:pi-point}
A \emph{$\pi$-point} of $\mcE$  consists of an extension field $K$ of $k$ together with a map of graded $K$-algebras
\[ 
\alpha \colon \test_K  \to K\mcE_K \quad\text{where $\test_K\colonequals K\otimes_k\test$},
\]
such that $K\mcE_K$ has finite flat dimension as an $\test_K$-module via $\alpha$. Since $\alpha$ is a map of $K$-algebras and $\rank_K K\mcE_K$ is finite, the latter is a finitely generated $\test_K$-module. In particular its flat dimension is the same as its projective dimension. 

Of particular interest are maps $\alpha_{\bsa}\colon \test_K \to K\mcE_K$ defined by the assignment
\begin{align}
\label{eq:alpha}
  \begin{split}
     \alpha_{\bsa}(t) &=  a_1 s_1 + \cdots + a_{n-1}s_{n-1} +  a_n s_n^{p^{m-1}} + a_{n+1}^2s_n\,, \\ 
      \alpha_{\bsa}(\tau)&=a_{n+1}^p\sigma \\
  \end{split}
\end{align}
for  $\bsa\colonequals (a_1,\dots,a_{n+1})\in \bbA^{n+1}(K)$. As we shall see in \intref{Proposition}{pr:reps} such a map is a $\pi$-point, that is to say, has finite flat dimension, if and only if  $\bsa\ne 0$.
\end{definition}

Our first task is to describe criteria that detect when a given map of $k$-algebras 
\[
\alpha\colon \test \to k\mcE
\]
has finite flat dimension.

\begin{construction}
\label{con:factorisation}
We keep the notation from \intref{Construction}{con:powerseries}. In particular, $k\mcE = P/I$ where $P=\pos{\bss,\sigma}$ is the power series over $k$ in indeterminates $\bss = s_1,\dots,s_n$ and $\sigma$, and $I$ is the ideal generated by the regular sequence 
\[
s_1^p,\dots,s_{n-1}^p, s_n^{p^m}, s_n^p-\sigma^2\,.
\]
It is helpful also to consider the following ideal and the $k$-algebra it determines
\begin{equation}
\label{eq:J}
J\colonequals (s_1^{p},\dots, s_{n-1}^p, s_n^{p^m}) \subset \pos{\bss} \quad\text{and}\quad S \colonequals \pos{\bss}/J\,.
\end{equation}
Observe that $S$ is a $k$-subalgebra of $k\mcE$ and the latter is free as an $S$-module, with basis $\{1,\sigma\}$.  A map of $k$-algebras $\alpha \colon \test \to k\mcE$ is thus determined by polynomials $f(\bss), g(\bss)$  in $k[\bss]$, uniquely defined only modulo $J$,  such that 
\begin{equation}
\label{eq:fg-condition}
f(\bss)^p   \equiv  g(\bss)^2 s_n^p      \quad \text{modulo} \quad J 
\end{equation}
where  $\alpha(t)= f(\bss)$ and $\alpha(\tau)= g(\bss)\sigma$.  Set
\[
P_{\alpha} \colonequals \frac{\pos{\bss,\sigma}}{(f(\bss)^p - g(\bss)^2\sigma^2)}\,.
\]
Observe that the element $(f(\bss)^p - g(\bss)^2\sigma^2)$ is in the ideal $I$ defining $k\mcE$; see \eqref{eq:fg-condition}. Thus $\alpha$ can be factored as
\begin{equation}
\label{eq:factor-alpha}
\test \xra{\ \dot\alpha\ } P_{\alpha} \xra{\ \bar{\alpha}\ } k\mcE
\end{equation}
where $\dot\alpha$ maps $(t,\tau)$ to $(f(\bss),g(\bss)\sigma)$ and $\bar\alpha$ is the canonical surjection.
\end{construction}

The result below contains criteria for detecting when a given map of $k$-algebras from $\test$ to $k\mcE$ is a $\pi$-point.

\begin{theorem}
\label{th:factorisation}
In the notation of \intref{Construction}{con:factorisation} assume that
$\alpha$ does not factor through the surjection $\test\to k$. Then the
map $\dot\alpha$ has finite flat dimension, and for any $k\mcE$-module
$M$ the module $\alpha^*(M)$ has finite flat dimension if and only if
${\bar\alpha}^*(M)$ has finite flat dimension.
\end{theorem}

\begin{proof}
As $\alpha$ does not factor through the surjection $\test\to k$, at
least one $f(\bss)$ or $g(\bss)$ is non-zero in $k\mcE$. This implies that $f(\bss)^p-g(\bss)^2\sigma^2$ is non-zero in $\pos{\bss,\sigma}$.

Consider the map of $k$-algebras
\[
\tilde\alpha \colon \pos{t,\tau} \lra \pos{\bss,\sigma} \quad\text{where}\quad t\mapsto f(\bss) \quad \text{and}\quad \tau\mapsto g(\bss)\sigma\,.
\]
This is  a lift of the map $\dot\alpha$, meaning that the following square commutes
\[
\begin{tikzcd}
\pos{t,\tau} \arrow[r, "{\tilde\alpha}"] \arrow[d,twoheadrightarrow] & \pos{\bss,\sigma} \arrow[d,twoheadrightarrow] \\
\test \arrow[r,"{\dot\alpha}"] & P_\alpha
\end{tikzcd}
\]
Here the vertical maps are the canonical surjections.  We claim that this diagram of $k$-algebras is a Tor-independent fibre square:
\[
\Tor^{\pos {t,\tau}}_i(\pos{\bss,\sigma},\test) = 
\begin{cases}
P_\alpha & \text{for $i=0$}\\
0  & \text{for $i\ne 0$.}
\end{cases}
\]
Indeed, the complex $0\to \pos{t,\tau}\xra{t^p-\tau^2} \pos{t,\tau}\to 0$ is a free resolution of $\test$ over $\pos{t,\tau}$, so the Tor-modules in question are the homologies of the complex
\[
\begin{tikzcd}
0 \arrow[r] & \pos{\bss,\sigma} \arrow[rr,"{f(\bss)^p-g(\bss)^2\sigma^2}"] && \pos{\bss,\sigma} \arrow[r] & 0\,.
\end{tikzcd}
\]
In degree zero the homology is evidently $P_\alpha$.  Since $f(\bss)^p-g(\bss)^2\sigma^2$ is non-zero the complex above has no homology in other degrees. This justifies the claim.

The ring $\pos{t,\tau}$ is regular, of Krull dimension two, so any map out of it, in particular $\tilde\alpha$, has finite flat dimension. And since the square above is a Tor-independent fibre square, we deduce that the flat dimension of $\dot\alpha$ is finite as well.

As to the second part of the statement: Since $\dot\alpha$ has finite flat dimension, when ${\bar{\alpha}}^*(M)$ has finite flat dimension, so does the $\test$-module $\alpha^*(M)={\dot\alpha}^*({\bar{\alpha}}^*(M))$. The proof of the converse exploits the following statement about the bounded derived category, $\dbcat{P_\alpha}$, of the hypersurface $P_\alpha$: 
\[
k \in \Thick(F) \quad \text{for $F\in\dbcat{P_\alpha}$  not perfect.}
\]
See \cite[Corollary~6.9]{Takahashi:2010a}. We will need this result  for the $P_\alpha$-complex
\[
F\colonequals P_\alpha\lotimes_{\test} k \,,
\]
where $P_\alpha$ acts through the left-hand factor of the tensor product. Observe that this $P_\alpha$-complex is in $\dbcat{P_\alpha}$ for one has isomorphisms
\begin{align*}
P_\alpha \lotimes_{\test} k 
	&\cong ( \pos{\bss,\sigma}  \lotimes_{\pos{t,\tau}} \test ) \lotimes_{\test} k \\
	&\cong  \pos{\bss,\sigma}  \lotimes_{\test} k
\end{align*}
in $\dbcat{ \pos{\bss,\sigma}}$. Moreover $F$ is not perfect because
\[
k\lotimes_{P_\alpha} F  = k\lotimes_{P_\alpha} (P_\alpha \lotimes_{\test} k) \cong k\lotimes_{\test} k
\]
which has non-zero homology in all non-negative degrees.

Now suppose that $\alpha^*(M)$ has finite flat dimension, so  $\Tor^{\test}_i(M,k)=0$ for $i\gg 0$. Since $\dot\alpha$ is flat, associativity of tensor products   yields isomorphisms
\[
M\lotimes_{\test} k \cong  M \lotimes_{P_{\alpha}} (P_\alpha \lotimes_{\test} k)   \cong M\lotimes_{P_{\alpha}} F\,.
\]
We deduce that
\[
\Tor^{P_{\alpha}}_i(M,F)=0 \quad \text{for $i\gg 0$.}
\]
As noted above, the $P_\alpha$-module $k$ is in $\Thick(F)$ in $\dcat{P_\alpha}$. It follows that
\[
\Tor^{P_{\alpha}}_i(M,k)=0 \quad \text{for $i\gg 0$}
\]
as well. Since $M$ has finite Loewy length, we conclude that the flat dimension of $\bar{\alpha}^*(M)$ is finite, as desired; see \intref{Lemma}{le:pd-test}.
\end{proof}

\begin{corollary}
\label{co:factorisation}
For any map of $k$-algebras $\alpha\colon A\to k\mcE$ the  following conditions are equivalent.
\begin{enumerate}[\quad\rm(1)]
\item
 $\alpha$ is a $\pi$-point;
 \item
 $\dot\alpha$ and $\bar{\alpha}$ are of finite flat dimension;
 \item
 $\bar{\alpha}$ is complete intersection of codimension $n$;
 \item
 $f(\bss)^p - g(\bss)^2\sigma^2$ is not in $\fm I$.
\end{enumerate}
\end{corollary}

\begin{proof}
(1)$\Leftrightarrow$(2) In either case, since the flat dimension of the $\test$-module $k$ is infinite, $\alpha$ cannot factor through $k$. Thus the desired equivalence holds by \intref{Theorem}{th:factorisation} applied to $M=k\mcE$.

\smallskip

(2)$\Rightarrow$(4) If $f(\bss)^p - g(\bss)^2\sigma^2$ is in $\fm I$,  the flat dimension of $\bar{\alpha}$ is infinite; this is by a result of Shamash~\cite[Theorem~1]{Shamash:1969a}. See also \cite[Theorem~2.1(3)]{Avramov/Iyengar:2018a}.

\smallskip

(3)$\Leftrightarrow$(4) The ideal $I$ is evidently a complete intersection and minimally generated by $n+1$ elements. Thus the hypothesis  in (3) is equivalent to the condition that the element $f(\bss)^p - g(\bss)^2\sigma^2$ can be extended to a minimal generating set for the ideal $I$; that is to say, it is not in $\fm I$.

\smallskip

(4)$\Rightarrow$(1) Any complete intersection map has finite flat dimension; this applies in particular to $\bar{\alpha}$. By the equivalence of (4) and (3), one has that $f(\bss)^p - g(\bss)^2\sigma^2$ is not in $\fm I$, and in particular, $\alpha$ does not factor through $k$. As already discussed above, this implies $\dot\alpha$ has finite flat dimension.
\end{proof}

\begin{remark}
In the factorisation \eqref{eq:factor-alpha} of a $\pi$-point, the map $\dot\alpha$ need not be flat. Indeed, it is easy to verify that that the map is flat if and only if the polynomials $f(\bss),g(\bss)$ can be chosen without non-trivial common factors; equivalently, that they form a regular sequence in the ring $\pos{\bss,\sigma}$.

Here is an example where this property does not hold: Fix an integer $m\ge 3$, an element $a\in k$ such that $a^2\ne 1$, and consider the $\pi$-point
\[
\alpha\colon \test \lra \frac{\pos{s,\sigma}}{(s^{p^m},s^p-\sigma^2)}  \quad
\text{where} \quad t\mapsto s^{p^{m-1}} \quad \text{and}\quad \tau \mapsto a s^{n}
\]
 for $n=(p^m-p)/2$.  That this is indeed a $\pi$-point can be easily checked using criterion (4) in \intref{Corollary}{co:factorisation}. 

On the other hand, for the ``standard" $\pi$-points $\alpha_{\bsa}$ described in \eqref{eq:alpha}, the maps $\dot\alpha_{\bsa}$ are flat, because $(f(\bss),g(\bss))=1$. This is noteworthy because, for our purposes, any $\pi$ point is equivalent to a standard one; see \intref{Proposition}{pr:reps}.
\end{remark}

The following lemma is intended for use in the proof of  \intref{Proposition}{pr:reps}.

\begin{lemma}
\label{le:reps}
Let $f(\bss),g(\bss)$ be polynomials in $k[\bss]$ satisfying equation \eqref{eq:fg-condition}. Then there exist uniquely defined elements $b_1,\dots,b_{n+1}$ in $k$ such that
\[ 
f(\bss)^p - g(\bss)^2\sigma^2  \equiv b_{1}s_1^{p} + \dots + b_{n-1} s_{n-1}^p + b_n s_n^{p^m} + b_{n+1} (s_n^p-\sigma^2)
\]
modulo the ideal $\fm I$  in  $k[\bss,\sigma]$, where $\fm$ and $I$ are as in \eqref{eq:mI}.
\end{lemma}

\begin{proof}
With $J$ as in \eqref{eq:J},  the action of $k[\bss]$ on $J/{\bss}J$ factors through the quotient $k[\bss]/(\bss)\cong k$, so $J/{\bss}J$  is a $k$-vector space with basis the residue classes of elements
\[
s_1^{p},\dots, s_{n-1}^p, s_n^{p^m}\,.
\]
Thus each element of $J$  is equivalent modulo ${\bss}J$ to an element of the form
\[
b_1 s_1^{p} +\cdots + b_{n-1}s_{n-1}^p + b_n s_n^{p^m}
\]
for uniquely defined elements $b_1,\dots,b_n$ in $k^n$. This applies in particular to the element $f(\bss)^p - g(\bss)^2s_n^p $, which is in $J$ by our hypothesis. Observe that in the ring $k[\bss,\sigma]$ the extended ideal ${\bss}Jk[\bss, \sigma]$ is contained in $\fm I$, for $\bss\subset \fm$ and $J\subset I$.  Therefore in the ring $k[\bss,\sigma]$ we get that
\begin{align*}
f(\bss)^p  - g(\bss)^2\sigma^2 
	&= f(\bss)^p - g(\bss)^2s_n^p  + g(\bss)^2(s_n^p- \sigma^2 )   \\
	&\equiv b_1 s_1^{p} + \cdots + b_{n-1} s_{n-1}^p +b_n s_n^{p^m} + g(\bss)^2(s_n^p- \sigma^2)  \\
	&\equiv b_1 s_1^{p} + \cdots + b_{n-1} s_{n-1}^p +b_n s_n^{p^m} + g(0)^2(s_n^p- \sigma^2) 
\end{align*}
where the equivalence is modulo $\fm I$; for the last step observe that ${\bss}(s_n^p-\sigma^2)\subset \fm I$.  Setting $b_{n+1}= g(0)^2$ completes the proof. 
\end{proof}

\begin{proposition}
\label{pr:reps}
Assume $k$ is closed under taking $p$th roots and square roots.  Let $\alpha\colon \test\to k\mcE_k$ be a map of  $k$-algebras and $(b_1,\dots,b_{n+1})\in \bbA^{n+1}(k)$  associated to $\alpha$  by \intref{Lemma}{le:reps}.   Set $a_i = b_i^{1/p}$ for $1\le i\le n$ and $a_{n+1} = b_{n+1}^{1/2p}$, and  $\alpha_{\bsa}\colon \test\to k\mcE$   the morphism of $k$-algebras defined in \eqref{eq:alpha}. The following statements hold.
\begin{enumerate}[\quad\rm(1)]
\item
The map $\alpha$ is a $\pi$-point if and only if $(a_1,\dots,a_{n+1})\ne 0$.
\item
When $\alpha$ is a $\pi$-point, and $M$ is a $k\mcE$-module, the following are equivalent:
\begin{enumerate}[\quad\rm(a)]
\item
$\alpha^*(M)$ has finite flat dimension;
\item
$\alpha_{\bsa}^*(M)$ has finite flat dimension.
\end{enumerate}
\end{enumerate}
\end{proposition}

\begin{proof}
\intref{Lemma}{le:reps} implies that $(b_1,\dots,b_{n+1})\ne 0$ holds if and only if $f(\bss)^p-g(\bss)^2\sigma^2$ is not in $\fm I$. Thus (1) is a consequence of \intref{Corollary}{co:factorisation}. 

Let $(f_{\bsa}(\bss),g_{\bsa}(\bss))$ be the polynomials defining the map $\alpha_{\bsa}$; thus
\begin{align*}
f_{\bsa}(\bss) &=  a_1 s_1  + a_{n-1}s_{n-1} +a_n  s_n^{p^{m-1}}  + \bsa^2_ns_n \\
g_{\bsa}(\bss) &= a_{n+1}^p\,.
\end{align*}
Therefore by the choice of the $a_i$ and \intref{Lemma}{le:reps} one has that
\[
 f(\bss)^p-g(\bss)^2\sigma^2 \equiv f_{\bsa}(\bss)^p-g_{\bsa}(\bss)^2\sigma^2  \pmod {\fm I}\,.
 \]
The desired result is thus a special case of \cite[Theorem~2.1(2)]{Avramov/Iyengar:2018a}.
\end{proof}

Next we record the following consequence of the preceding results; this is the main outcome of this section, as far as the sequel is concerned.

\begin{theorem}
\label{th:detection}
Let $M$ be a $k\mcE$-module. The following conditions are equivalent:
\begin{enumerate}[\quad\rm(1)]
\item
$M$ is projective;
\item
$\alpha^*(M_K)$ has finite flat dimension for every $\pi$-point $\alpha\colon \test_K\to K\mcE_K$;
\item
$\alpha^*(M^K)$ has finite flat dimension for every $\pi$-point $\alpha\colon \test_K\to K\mcE_K$.
\end{enumerate}
\end{theorem}

\begin{proof}
  For any extension field $K$ of $k$ and
  $\bsa = (a_1,\dots,a_{n+1}) \in \bbA^{n+1}(K)$, and for the
  $\pi$-point $\alpha_{\bsa}\colon \test_K\to K\mcE$ described by
  \eqref{eq:alpha}, one has that $\alpha_{\bsa}^*(M_K)$, respectively
  $\alpha_{\bsa}^*(M^K)$, has finite flat dimension if and only if
  $(a_1^p,\dots,a_n^p,a_{n+1}^{2p})$ is in $\hysupp_{\mcE_K}(M_K)$,
  respectively, in $\hysupp_{\mcE_K}(M^K)$. This claim is a direct
  consequence of \intref{Proposition}{pr:reps}(2).  Thus the desired
  equivalence follows from \intref{Theorem}{th:support-sets}.
\end{proof}

\section{The rank variety}
\label{se:rank}
In this section, we restrict our attention to finitely generated $k\mcE$-modules, and describe the analogue of Carlson's rank variety~\cite{Carlson:1983a}, relating it to support sets introduced in \intref{Section}{se:supportsets}. The field $k$, always of positive characteristic $p\ge 3$, will be algebraically closed.

\begin{definition}
For a point $\bsa = (a_1,\dots,a_{n+1})$ in $\bbA^{n+1}(k)$  let $\alpha_{\bsa} \colon \test\to k\mcE$ be the $\pi$-point defined in \eqref{eq:alpha}. 

Let $M$ be a finitely generated $k\mcE$-module and consider the square
zero  map
 \begin{equation}
\label{eq:Mmatrix}
  \begin{pmatrix} 
  \alpha_{\bsa}(\tau) & \alpha_{\bsa}(t) \\ 
  -\alpha_{\bsa}(t)^{p-1} &    -\alpha_{\bsa}(\tau) 
  \end{pmatrix} \colon M\oplus M \lra M\oplus M\,.
 \end{equation} 
The rank of this map is at most that of $M$; see \intref{Theorem}{th:rank}. The \emph{rank variety} of $M$, denoted $V^r_\mcE(M)$, is the subset of $\bbA^{n+1}(k)$ consisting those $\bsa$ for which the map in \eqref{eq:Mmatrix} fails to have maximal rank, in the sense of \intref{Definition}{de:maximal-rank}. In particular,  $0$ is in the rank variety of any $M$.
\end{definition}

\begin{remark}
In contrast with support sets the rank varieties are introduced as subsets of affine space rather than projective space. Observe that swapping the diagonal entries in the matrix \eqref{eq:Mmatrix} yields the matrix corresponding to the point $(a_1,\dots, -a_{n+1})$. In other words, the rank variety on $M$ is invariant under the $\bbZ/2$ action on the last coordinate that sends $a_{n+1}$ to $-a_{n+1}$.  The intervention of this automorphism is not surprising. It plays a role in Deligne's work on Tannakian categories~\cite[Section~8]{Deligne:1990a}, where it is the generator for the fundamental group of the category of super vector spaces.  In his description~\cite{Deligne:2002a} of symmetric tensor abelian categories of moderate growth in characteristic zero, this central involution is an essential ingredient.

There is more: Consider the action of $k^\times$ on $\bbA^{n+1}(k)$ given by
\begin{equation}
\label{eq:k-action}
\lambda(a_1,\dots,a_{n+1})\colonequals (\lambda^2 a_1,\dots, \lambda^2 a_n, \lambda a_{n+1})
\end{equation}
and the map $F\colon \bbA^{n+1}(k)\to \bbP^{n}(k)$ where 
\[
(a_1,\dots,a_{n+1}) \mapsto [a_1^p,\dots, a_n^p,a_{n+1}^{2p}]\,.
\]
Observe that $F(\lambda \bsa) = \lambda^{2p} F(\bsa)$. This observation is used in the proof of the result below, in which part (1) compares the rank variety of $M$ to its support set~\eqref{eq:Av-variety}.
\end{remark}

\begin{theorem}
\label{th:carlson}
  Let $M$ be a finitely generated $k\mcE$-module.
  \begin{enumerate}[\quad\rm(1)]
  \item
   The rank variety $V^r_\mcE(M)$ is a closed subset of $\bbA^{n+1}(k)$, homogeneous for the action of $k^\times$ on $\bbA^{n+1}$ defined in \eqref{eq:k-action}.
  \item
  A point $\bsa\in \bbA^{n+1}(k)\setminus\{0\}$ is in $V^r_{\mcE}(M)$ if and only if $F(\bsa)$ is in $\hysupp_{\mcE}(M)$. 
   \item
  The $k\mcE$-module $M$ is projective if and only if $V^r_\mcE(M)=\{0\}$.
     \end{enumerate}
\end{theorem}

\begin{proof}
The first part of (2) is immediate from the description of the two
varieties and \intref{Proposition}{pr:reps}(2).
%The map $F$ is bijective because $k$ is algebraically closed. 

(1) The condition for a point $\bsa$ to be in $V^r_\mcE(M)$ is given by the vanishing of the minors of the matrix \eqref{eq:Mmatrix} of size equal to $\rank_kM$. These are polynomial relations in the coordinates, so $V^r_{\mcE}(M)$ is closed under the Zariski topology on $\bbA^{n+1}(k)$.  

Fix $\bsa$ in $V^r_{\mcE}(M)\setminus\{0\}$, so that $F(\bsa)$ is in $\hysupp_{\mcE}(M)$, by (2).  We deduce that the point $F(\lambda\bsa) = \lambda^{2p}F(\bsa)$ is in $\hysupp_{\mcE}(M)$ as well. Thus $\lambda\bsa$ is in $V^r_{\mcE}(M)$, again by (2). This justifies the assertion that the rank variety is homogeneous.

(3) The map $F$ is onto because the field $k$ is algebraically
closed. Thus (3) is immediate from (2) and the fact that $\hysupp_{\mcE}(M)=\varnothing$ if and only if $M$ is projective; see \intref{Theorem}{th:support-sets}.
\end{proof}

\begin{remark}
The matrix \eqref{eq:Mmatrix} corresponding to $\lambda \bsa$ is 
\[
  \begin{pmatrix} 
  \lambda^p\alpha_{\bsa}(\tau) & \lambda^2\alpha_{\bsa}(t) \\ 
  -\lambda^{2p-2}\alpha_{\bsa}(t)^{p-1} &    -\lambda^p\alpha_{\bsa}(\tau) 
  \end{pmatrix} 
\]
It follows from \intref{Theorem}{th:carlson}(2) that this matrix has
maximal rank if and only if the one in \eqref{eq:Mmatrix} does. Thus
the subset $V^r_{\mcE}(M)$ of $\bbA^{n+1}(k)$ is homogenous for the
$k^{\times}$-action in \eqref{eq:k-action}. A direct proof of this
observation seems complicated.
% cf.~ \cite[??]{Bendel/Friedlander/Suslin:1997a}
\end{remark}

\section{$\pi$-points and cohomology}
\label{se:pi-and-cohomology}
In this section we set up the equivalence relation on $\pi$-points for $k\mcE$ and the bijection between their equivalence classes and points in the weighted projective space $\Proj H^{*,*}(\mcE,k)$. The development is modelled on the one for group schemes due to Friedlander and Pevtsova~\cite{Friedlander/Pevtsova:2007a}.

Given  a $\pi$-point $\alpha\colon \test_K\to K\mcE$ the restriction functor 
\[
\alpha^*\colon \Mod(K\mcE_K) \to \Mod(\test_K)
\]
takes projective modules to modules of finite flat dimension. Composed with the approximation functor, it induces an exact functor 
\begin{equation}
\label{eq:alpha*}
\alpha^*\colon \StMod(K\mcE_K)\lra \uGProj(\test_K). 
\end{equation}
This need not be compatible with the tensor structures, for $\alpha$ need not be a morphism of Hopf algebras.

\begin{definition} 
If $\alpha\colon \test_K \to K\mcE_K$ is a $\pi$-point of $\mcE$, we write $H^{*,*}(\alpha)$ for the composition of homomorphism of $k$-algebras
\[ 
H^{*,*}(\mcE,k) = \Ext^{*,*}_{k\mcE}(k,k) \xrightarrow{K \otimes_k -}
\Ext^{*,*}_{K\mcE_K}(K,K) \xrightarrow{\alpha^*} \Ext^{*,*}_{\test_K}(K,K) 
\]
The cohomology of the algebra $\test$ was described in \eqref{eq:Acohomology}; modulo its radical it is a domain. Thus the radical of the kernel of $H^{*,*}(\alpha)$ is a prime ideal, denoted $\fp(\alpha)$. 
\end{definition}

\begin{example}
Fix $\bsa \in \bbA^{n+1}(k)$ and let $\alpha_{\bsa}$ be the $\pi$-point described in \eqref{eq:alpha}.  In the notation of  \eqref{eq:Ecohomology} and \eqref{eq:Acohomology}, the  map induced in cohomology by $\alpha_{\bsa}$ is given by 
\begin{alignat*}{4}
&\alpha^*_{\bsa}(u_i) = a_i v&  \quad &\text{for $1\le i\le n-1$}& \quad\text{and}\quad & \alpha^*_{\bsa}(u_n) = a_{n+1}^2 v \\
&\alpha^*_{\bsa}(x_i) = a_i^p \eta^2& \quad &\text{for $1\le i\le n$}&\quad\text{and}\quad   &  \alpha^*_{\bsa}(\zeta) = a_{n+1}^p \eta.
\end{alignat*}
Therefore the associated point is $\Proj H^*(\mcE,k)$ is
\begin{equation}
\label{eq:alpha-prime}
\fp(\alpha_{\bsa}) = 
\begin{cases}
\left(a^{2p}_{n+1}x_i - a^p_i \zeta^2\mid 1\le i \le n\right)+(u_1,\cdots,u_n) & \text{if $a_{n+1}\ne 0$}\\
(\zeta)+ \left(a^p_i x_j -a^p_j x_i\mid 1\le i < j\le n\right)+(u_1,\cdots,u_n)   & \text{if $a_{n+1}= 0$.}
\end{cases}
\end{equation}
Compare this with \eqref{eq:beta-prime}. In particular, when $k$ is closed under taking  square roots and $p$th roots (for example, if $k$ is algebraically closed) each rational  point in $\Proj H^*(\mcE,k)$ occurs as $\fp(\alpha)$, for some $\pi$-point $\alpha$. Here is the general statement.
\end{example}

\begin{proposition}
Given a point $\fp$ in $\Proj H^{*,*}(\mcE,k)$, there exists a field extension $K$ of $k$ and a $\pi$-point $\alpha\colon \test_K\to K\mcE_K$ of the
form~\eqref{eq:alpha} such that $\fp(\alpha)=\fp$.
\end{proposition}

\begin{proof}
 Choose a perfect field extension $K$ of $k$  containing the function field for $\fp$, that is, the degree $(0,0)$ elements of the graded field of fractions of $H^{*,*}(\mcE,k)/\fp$. Then there is a closed point $\fm\in \Proj H^{*,*}(\mcE_K,K)$ lying over $\fp$. In coordinates, this is given by $[b_1,\cdots,b_{n+1}]$  with not all the $b_i$ equal to zero. Setting $a_i=b_i^{1/p}$, the map $\alpha_{\bsa}$ defined by \eqref{eq:alpha} has $\fp(\alpha_{\bsa})=\fp$.
\end{proof}

The following definition is the analogue of~\cite[2.1]{Friedlander/Pevtsova:2007a}.

\begin{definition}
We say that $\pi$-points $\alpha\colon \test_K\to K\mcE_K$ and $\beta\colon \test_L\to L\mcE_L$  are \emph{equivalent}, and write  $\alpha\sim\beta$, if for each finitely generated $k\mcE$-module $M$, the $\test_K$-module $\alpha^*(M_K)$ has finite flat dimension if and only if the $\test_L$-module $\beta^*(M_L)$ does. 
\end{definition}

Other ways of expressing this equivalence relation are given in \intref{Theorem}{th:equiv}. In preparation for this, we recall the definition of Carlson's modules. For any non-zero element $\xi\in H^{2n,0}(\mcE,k)$ let $L_\xi$ be the kernel of a representative cocycle $\hat \xi\colon \Omega^{2n}(k) \to k$. Thus we have an exact sequence of $k\mcE$-modules
\begin{equation}
\label{eq:Carlson}
0 \to L_\xi \to \Omega^{2n}(k) \xrightarrow{\hat\xi} k \to 0 
\end{equation}
and a corresponding triangle $L_\xi \to \Omega^{2n}(k) \to k\to$ in $\StMod(k\mcE)$. The following result is analogous to \cite[Proposition~2.3]{Friedlander/Pevtsova:2005a} and \cite[Proposition~2.9]{Friedlander/Pevtsova:2007a}.

\begin{proposition}
\label{pr:Lxi}
Let $\alpha\colon \test_K\to K\mcE_K$ be a $\pi$-point. An element $\xi\in H^{2n,0}(\mcE,k)$ is not in $\fp(\alpha)$ if and only if $\alpha^*(L_\xi)_K$ has finite flat dimension. If $\alpha$, $\beta$ are $\pi$-points with  $\fp(\alpha)\ne \fp(\beta)$ then $\alpha$ and $\beta$ are inequivalent.
\end{proposition}

\begin{proof}
Applying $(-)_K$ to the exact triangle defining $L_\xi$ yields a triangle 
\[
(L_\xi)_K \lra \Omega^{2n}(K) \lra K\lra
\]
in $\StMod(K\mcE_K)$. Apply $\alpha^*$ to this triangle, and keep in mind ~\eqref{eq:alpha*},  to get a triangle 
\[
\alpha^*(L_\xi)_K \lra \Omega^{2n}(\alpha^*(K)) \lra \alpha^*(K)\lra
\]
 in $\uGProj(\test)$. Now $H^{2n,0}(\test,K)$ is two dimensional, but its image in $H^{*,*}(\test,K)$ modulo $\sqrt{0}$ is  one dimensional. Elements outside $\sqrt{0}$ induce an isomorphism  from $\Omega^{2n}(\alpha^*(K))$  to $\alpha^*(K)$ in $\uGProj(\test)$.  So  $\alpha^*(L_\xi)_K$ has finite flat dimension if and only if it  is zero in $\uGProj(\test)$, which happens if and only if  $\alpha^*(\xi)$ is non-zero in $H^{2n,0}(\test,K)$ modulo $\sqrt{0}$, and this happens  if and only if $\xi\not\in\fp(\alpha)$.

If $\fp(\alpha)\ne \fp(\beta)$,  choose $\xi\in H^{2n,0}(\mcE,k)$ in $\fp(\beta)$ but not in $\fp(\alpha)$. Then $\beta^*(L_\xi)_L$ has infinite flat dimension while $\alpha^*(L_\xi)_K$ has finite flat dimension. Since $L_\xi$ is finitely generated,  it follows that $\alpha$ and $\beta$ are inequivalent.
\end{proof}

\begin{theorem}
\label{th:equiv}
Let $\mcE$ be an elementary supergroup scheme. Let $\alpha\colon \test_K\to K\mcE_K$, $\beta\colon \test_L\to L\mcE_L$ be $\pi$-points of $\mcE$. Then the following are  equivalent:
\begin{enumerate}[\quad\rm(1)]
\item 
$\alpha$ and $\beta$ are equivalent.
\item 
For all $k\mcE$-modules $M$,  the $\test_K$-module $\alpha^*(M_K)$ has finite flat dimension if and only if the $\test_L$-module  $\beta^*(M_L)$ has finite flat dimension.
\item 
For all $k\mcE$-modules $M$,  the $\test_K$-module $\alpha^*(M^K)$ has finite flat dimension if and only if the $\test_L$-module  $\beta^*(M^L)$ has finite flat dimension.
\item 
$\fp(\alpha)=\fp(\beta)$.
\end{enumerate}
Thus the map sending $\alpha$ to $\fp(\alpha)$ induces a bijection between the set of equivalence classes of $\pi$-points and the set $\Proj H^{*,*}(G,k)$.
\end{theorem}

\begin{proof}
It is clear that (2) and (3), even limited  to finitely generated
modules, imply (1). \intref{Proposition}{pr:Lxi} shows that (1)
implies (4).  In the rest of the proof we suppose that (4) holds, and
verify that (2) and (3) do. By extending fields if necessary, we may assume that $k$ is algebraically closed, and that $L=k=K$.  

We can also assume that $\alpha=\alpha_{\bsa}$ and $\beta=\alpha_{\bsb}$, as in \eqref{eq:alpha}, for some $\bsa,\bsb\in \bbA^{n+1}(k)$. 

Indeed, \intref{Proposition}{pr:reps} yields  that $\alpha$ and $\beta$ are equivalent as $\pi$-points to $\alpha_{\bsa}$ and $\alpha_{\bsb}$, respectively, for some $\bsa,\bsb$.  Moreover $\fp(\alpha) = \fp(\alpha_{\bsa})$ and $\fp(\beta)=\fp(\alpha_{\bsb})$, for we already know (1)$\Rightarrow$(4). This justifies the assumption.

It is  clear from \eqref{eq:alpha-prime} that $\fp(\alpha_{\bsa})=\fp(\alpha_{\bsb})$ implies
\[
(a_1^p,\dots,a_n^p,a_{n+1}^{2p})= (b_1^p,\dots,b_n^p,b_{n+1}^{2p})\,.
\]
Then, in the notation of \eqref{eq:factor-alpha} one gets that $\bar{\alpha}_{\bsa} = \bar{\alpha}_{\bsb}$, and  so \intref{Theorem}{th:factorisation} implies that (2) and (3) hold.
\end{proof}

\section{Support and cosupport}
\label{se:support-and-cosupport}
In this section we define the $\pi$-support and $\pi$-cosupport of a $k\mcE$-module $M$, and the cohomological support and cosupport,
and prove that these two notions of support, and of cosupport, coincide. Once this is done \cite{Benson/Iyengar/Krause/Pevtsova:2017a} provides a path to the classification of the localising subcategories of $\StMod k\mcE$.

\begin{definition}
Let $M$ be a $k\mcE$-module. The \emph{$\pi$-support} of $M$, denoted $\pisupp_\mcE(M)$, is the subset of  $\Proj H^*(\mcE,k)$ consisting of the primes $\fp(\alpha)$ where  $\alpha\colon\test_K\to K\mcE_K$ is a $\pi$-point such that  the flat dimension of $\alpha^*(M_K)$ is infinite. Replacing 
$\alpha^*(M_K)$ by $\alpha^*(M^K)$ gives the \emph{$\pi$-cosupport} of $M$, denoted $\picosupp_\mcE(M)$.
\end{definition}

\begin{remark}
For a finitely generated $k\mcE$-module $M$, we have 
\[
\pisupp_\mcE(M)=\picosupp_\mcE(M)\,.
\]
This is usually not true for infinitely generated modules; see \cite[Example~4.7]{Benson/Iyengar/Krause/Pevtsova:2017a}.
\end{remark}

The result below is now only a reformulation of \intref{Theorem}{th:detection}.

\begin{theorem}
\label{th:pi-proj}
 Let $M$ be a $k\mcE$-module. The following conditions are equivalent.
  \begin{enumerate}[\quad\rm(1)]
  \item $M$ is projective,
  \item $\pisupp_\mcE(M)=\varnothing$,
  \item $\picosupp_\mcE(M)=\varnothing$. \qed
  \end{enumerate}
\end{theorem}

The following proposition exhibits the effect of changing the Hopf
structure, in order to prove formulas for the support of tensor product
and Hom of two representations.

\begin{proposition}
\label{pr:superFP}
Let $M$ be a $k$-vector space and let $\beta, \gamma,\delta,\ve \colon M\to M$ be commuting $k$-linear maps such that 
\[
\gamma^{p}=0= \delta^{p^m}   \quad \text{and}\quad \delta^p = \ve^2
\]
for some integer $m\ge 1$. Then the map 
\[
 \begin{pmatrix} 
 \ve & \delta \\ 
 -\delta^{p-1} &    -\ve 
\end{pmatrix} \colon M\oplus M \lra M \oplus M 
 \]
is square zero. It has maximal image if and only if the map
\[ 
\begin{pmatrix} 
\ve & \delta+ \beta \gamma  \\
   -(\delta+ \beta \gamma)^{p-1} &    -\ve 
\end{pmatrix}
 \col M\oplus M \lra M \oplus M
 \]
 does.
\end{proposition}

\begin{proof}
We managed, not without some difficulty, to adapt the proof of~\cite[2.2]{Friedlander/Pevtsova:2005a} to this context. Here is an alternative approach, suggested by \cite[\S5.4]{Avramov/Iyengar:2018a} that makes it clear that what we seek is a variation on \intref{Proposition}{pr:reps}. Akin to the proof of \cite[5.6]{Avramov/Iyengar:2018a} we consider the $k$-algebra
\[
R \colonequals \frac{\pos{b,c,d,e}}{(c^{p}, d^{p^m}, d^p-e^2)}
\]
and maps of $k$-algebras
\[
\begin{tikzcd}
\displaystyle{\frac{k[t,\tau]}{(t^p-\tau^2)}} \arrow[r,yshift=0.7ex,"{\alpha_1}"] \arrow[r,yshift=-0.7ex,swap,"{\alpha_2}"] 
	& R  
\end{tikzcd} \quad\text{where} \quad
\begin{aligned}
&\alpha_1(t) = d & \text{ and } \alpha_1(\tau) = e \\
&\alpha_2(t) = d + bc & \text{ and } \alpha_2(\tau) = e. 
\end{aligned}
\]
Given \intref{Theorem}{th:rank}, the desired result is then that $\alpha_1^*(M)$ has finite flat dimension if and only if $\alpha_2^*(M)$ does. In the notation of \intref{Theorem}{th:factorisation} the pertinent hypersurfaces in $\pos{b,c,d,e}$ are defined by polynomials 
\[
d^p -e^2 \quad\text{and} \quad (d+bc)^p - e^2\,.
\]
Evidently these polynomials are congruent modulo $(b,c,d,e)(c^p,d^{p^m},d^p-e^2)$. This justifies the second equivalence below:
\begin{align*}
\alpha_1^*(M) \text{ has finite flat dimension } 
	& \iff \bar{\alpha}_1^*(M)  \text{ has finite flat dimension } \\
       & \iff \bar{\alpha}_2^*(M)  \text{ has finite flat dimension } \\
       &\iff \alpha_2^*(M) \text{ has finite flat dimension } 
\end{align*}
The first and the last one are by \intref{Theorem}{th:factorisation}.
\end{proof}

The result below is the analogue of~\cite[Lemma~3.9]{Friedlander/Pevtsova:2005a} and ~\cite[Lemma~4.3]{Benson/Iyengar/Krause/Pevtsova:2017a}.

\begin{lemma}
\label{le:tensor-hom-fpd}
Let $\alpha\colon \test\to K\mcE_K$ be a $\pi$-point and let $M,N$ be $K\mcE_K$-modules. The following conditions are equivalent:
\begin{enumerate}[\quad\rm(1)]
\item $\alpha^*(M\otimes_K N)$ has finite flat dimension.
\item $\alpha^*(\Hom_K(M,N))$ has finite flat dimension.
\item $\alpha^*(M)$ or $\alpha^*(N)$ has finite flat  dimension.    
\end{enumerate}
\end{lemma}

\begin{proof}
If $\alpha$ were a homomorphism of Hopf algebras, then the induced restriction functor $\alpha^*$ from \eqref{eq:alpha*} is a tensor functor, so combining it with  \intref{Proposition}{pr:detect} one gets tensor functors
\[ 
\StMod(K\mcE_K)\xrightarrow{\alpha^*} \uGProj(\test) \simeq  \StMod(K[t]/(t^p)). 
\]
In $\StMod(K[t]/(t^p))$, the tensor product or Hom of two modules is projective (i.e., zero) if and only if one of them is projective, which proves the lemma in this case.

The general case is tackled as in the proof of \cite[3.9]{Friedlander/Pevtsova:2005a}: By \intref{Proposition}{pr:reps}, we may assume that $\alpha$ has the
form~\eqref{eq:alpha}. We may therefore change the Hopf algebra structure on $K\mcE_K$ so that $\alpha$ is a homomorphism of Hopf algebras,
by making all the $s_i$ and $\sigma$ primitive:
\[
\Delta(s_i)=s_i\otimes 1 + 1 \otimes s_i \quad\text{and}\quad
\Delta(\sigma)=\sigma\otimes 1 + 1 \otimes \sigma\,.
\]
For any element $x$ of the augmentation ideal $I$ of $K\mcE_K$ this changes $\Delta(x)$ by an element of $I \otimes I$.

Now consider the action on $(M \otimes_K N) \oplus (M \otimes_K N)$ of the matrix 
\[
\begin{pmatrix} \tau & t \\
-t^{p-1} & -\tau
\end{pmatrix}
\]
via $\alpha$. The effect of changing from the old to the new diagonal
on this action  is a sequence of changes where $t\otimes 1 + 1 \otimes
t$ is replaced by $t\otimes 1 + 1 \otimes t + u \otimes v$ with
$u^p=0=v^p$. Applying \intref{Proposition}{pr:superFP}, we see that
this does not affect the maximal image property. So by \intref{Theorem}{th:rank}, it does not change whether $M \otimes_K N$ has finite projective
dimension. The argument for $\Hom_K(M,N)$ is similar.
\end{proof}

\begin{theorem}
\label{th:pi-tensor-hom}
Let $M$ and $N$ be $k\mcE$-modules. Then the following holds.
\begin{enumerate}[\quad\rm(1)]
\item
$\pisupp_\mcE(M\otimes_k N) =  \pisupp_\mcE(M)\cap \pisupp_\mcE(N)$.
\item
$\picosupp_\mcE(\Hom_k(M,N)) =
\pisupp_\mcE(M)\cap \picosupp_\mcE(N)$.
\end{enumerate}
\end{theorem}

\begin{proof}
  For (i), we use the isomorphism
  \[ 
  M_K \otimes_K N_K \cong (M \otimes_k N)_K. 
\]
  For (ii), we use the isomorphism
  \begin{equation*}
    \Hom_k(M,N)^K \cong \Hom_K(M_K,N^K).
  \end{equation*}
  In both cases, we then use \intref{Lemma}{le:tensor-hom-fpd}.
\end{proof}

Next we introduce cohomological support and cosupport, following \cite{Benson/Iyengar/Krause:2009a,Benson/Iyengar/Krause:2012b}. Recall that the stable category $\StMod(k\mcE)$ has two gradings: an internal one with shift $\Pi$ and a cohomological one with shift $\Omega^{-1}$, and a natural isomorphism $\Pi\Omega\cong\Omega\Pi$. Its
centre $Z\StMod(k\mcE)$ is doubly graded, and $Z^{i,j}\StMod(k\mcE)$ consists of those natural transformations $\gamma$ from the identity to $\Omega^{-i}\Pi^{j}$ which satisfy $\gamma\Omega=-\Omega\gamma$ and $\gamma\Pi=-\Pi\gamma$. This is a doubly graded commutative ring, in the sense that if $x\in Z^{i,j}$ and $y\in Z^{i',j'}$ then
\[ 
yx = (-1)^{(i+j)(i'+j')}xy. 
\]
The category $\StMod(k\mcE)$ is $H^{*,*}(\mcE,k)$-linear, in the sense of \cite{Benson/Iyengar/Krause:2009a}. Namely, there is a homomorphism of doubly graded $k$-algebras $H^{*,*}(\mcE,k) \to Z\StMod(k\mcE)$, given in the obvious way. For each $\fp\in\Proj H^{*,*}(\mcE,k)$ there is a local cohomology functor 
\[
\Gamma_\fp\colon\StMod(k\mcE)\to\StMod(k\mcE)
\]
and a local homology functor 
\[
\Lambda_\fp\colon\StMod(k\mcE)\to\StMod(k\mcE)
\]
defined in terms of this action; see \cite{Benson/Iyengar/Krause:2009a,Benson/Iyengar/Krause:2012b}.

\begin{definition}
The \emph{support} and \emph{cosupport} of a $k\mcE$-module $M$ are the subsets
  \begin{align*}
    \supp_\mcE(M)&\colonequals \{\fp\in\Proj H^{*,*}(\mcE,k) \mid \Gamma_\fp(M)\ne  0 \} \\
    \cosupp_\mcE(M)& \colonequals  \{\fp\in\Proj H^{*,*}(\mcE,k) \mid \Lambda_\fp(M)  \ne 0 \}. 
  \end{align*}
See \cite[\S5]{Benson/Iyengar/Krause:2009a} and  \cite[\S4]{Benson/Iyengar/Krause:2012b}. 
\end{definition}

\begin{proposition}
\label{pr:pi-Gamma-p}
We have $\pisupp(\Gamma_\fp(k))=\{\fp\}$.
\end{proposition}
\begin{proof}
  This statement is independent from the Hopf structure of $k\mcE$,
  since the local cohomology functors $\Gamma_\fq$ do not change when
  the Hopf structure is changed; see
  \cite[Corollary~3.3]{Benson/Iyengar/Krause:2011b}.

Thus we can use a Hopf structure with the property that the $\pi$-points of the form \eqref{eq:alpha} are Hopf maps $\alpha\colon\test\to K\mcE_K$. 

If $N$ is a finite dimensional $k\mcE$-module  then consider  the commutative diagram of finite maps
\[ 
\xymatrix{H^{*,*}(K\mcE_K,K) \ar[r]^{\alpha^*}\ar[d]_{-\otimes N} & 
H^{*,*}(\test,K)\ar[d]^{-\otimes\alpha^*(N)}  \\ 
\Ext^{*,*}_{K\mcE_K}(N,N) \ar[r]^(0.44){\alpha^*} &
    \Ext^{*,*}_{\test}(\alpha^*(N),\alpha^*(N)).} 
 \]
If $\fp(\alpha)$ is in the $\pi$-support of $N$ then the kernel of the top map followed by the right hand map is not contained in $\fp(\alpha)$. So the kernel of the left hand map is not contained in $\fp(\alpha)$ and hence $\fp(\alpha)\in\supp(N)$. Thus we have $\pisupp(N)\subseteq\supp(N)$. 

For a subset $\fa\subseteq H^{*,*}(\mcE,k)$ let $\mcV(\fa)$ denote the
set primes $\fp\in\Proj H^{*,*}(\mcE,k)$ such that $\fa\subseteq\fp$.

Since the
module $\Gamma_\fp(k)$ is a filtered colimit of finite dimensional
modules whose support is contained in $\mcV(\fp)$, it follows that
$\pisupp(\Gamma_\fp(k))\subseteq\mcV(\fp)$.

Now consider the Carlson module $L_\xi$ for an element
$\xi\in H^{2n,0}(\mcE,k)$, see \eqref{eq:Carlson}.
\intref{Proposition}{pr:Lxi} shows that $\pisupp(L_\xi)=\supp(L_\xi)$
equals the set of primes containing $\xi$. By
\intref{Theorem}{th:pi-tensor-hom}, it follows that if we have a
sequence of such elements, $\xi_1,\dots,\xi_m$ then
\[ 
\pisupp(L_{\xi_1}\otimes \dots\otimes L_{\xi_m}) = \mcV(\xi_1)\cap\dots\cap \mcV(\xi_m)=\mcV(\xi_1,\dots,\xi_m), 
\]
which is also equal to $\supp(L_{\xi_1}\otimes \cdots \otimes L_{\xi_m})$. So there are finite dimensional modules with any given closed subset
for both its $\pi$-support and its support.  If $N$ is such a module whose support is properly contained in $\mcV(\fp)$ then $\Gamma_\fp(k)\otimes N =
\Gamma_\fp(N)$ is projective, and so again using \intref{Theorem}{th:pi-tensor-hom},
we have
\[ 
\pisupp(\Gamma_\fp(k))\cap\pisupp(N)=\varnothing\,. 
\] 
This shows that $\pisupp(\Gamma_\fp(k))\subseteq\{\fp\}$. Since $\Gamma_\fp(k)$ is not projective, it now follows from \intref{Theorem}{th:pi-proj} 
that $\pisupp(\Gamma_\fp(k))=\{\fp\}$.
\end{proof}

\begin{theorem}
\label{th:pi-equals-H}
Let $\mcE$ be an elementary supergroup scheme and $M$ a $k\mcE$-module. Then $\cosupp_\mcE(M)=\picosupp_\mcE(M)$ and  $\supp_\mcE(M)=\pisupp_\mcE(M)$.
\end{theorem}

\begin{proof}
Using \intref{Proposition}{pr:pi-Gamma-p}, \intref{Theorem}{th:pi-proj} and \intref{Theorem}{th:pi-tensor-hom}, the proof is exactly the same as 
that of  \cite[Theorem~6.1]{Benson/Iyengar/Krause/Pevtsova:2017a}.
\end{proof}

\begin{corollary} 
For all $k\mcE$-modules $M$ and $N$ we have
\begin{enumerate}[\quad\rm(1)]
\item $\supp_\mcE(M \otimes_k N) = \supp_\mcE(M)\cap\supp_\mcE(N)$,
\item $\cosupp_\mcE(\Hom_k(M,N)) = \supp_\mcE(M) \cap\cosupp_\mcE(N)$.
\end{enumerate}
\end{corollary}

\begin{proof}
This follows from \intref{Theorem}{th:pi-equals-H} and \intref{Theorem}{th:pi-tensor-hom}.
\end{proof}

Given these expressions for the (cohomological) support and cosupport  one can mimic the proof of  \cite[Theorem~7.1]{Benson/Iyengar/Krause/Pevtsova:2017a} to deduce:

\begin{theorem}
\label{th:localising}
Let $\mcE$ be an elementary supergroup scheme over $k$. The stable module category $\StMod(k\mcE)$ is stratified as a $\bbZ/2$-graded  triangulated category by the natural action of the  cohomology ring $H^{*,*}(\mcE,k)$. Therefore the assignment
\[ 
\bfC \mapsto \bigcup_{M\in\bfC}\supp_\mcE(M) 
\]
gives a one to one correspondence between the localising subcategories of $\StMod(k\mcE)$ invariant under the parity change operator $\Pi$ and the subsets of $\Proj H^{*,*}(\mcE,k)$. \qed
\end{theorem}

From the preceding result, and following the by now well-trodden path
discovered by Neeman~\cite{Neeman:1992a} one obtains a classification
of thick subcategories of $\stmod k\mcE$; see also the proof of  \cite[Theorem~7.3]{Benson/Iyengar/Krause/Pevtsova:2017a}.

\begin{corollary}
\label{co:thick}
Let $\mcE$ be an elementary supergroup scheme over $k$. There is a one to one correspondence between  thick subcategories of
$\stmod(k\mcE)$ invariant under $\Pi$  and the specialisation closed subsets of $\Proj H^{*,*}(\mcE,k)$. \qed
\end{corollary}

\section{The remaining cases}
\label{se:others}
The gist of this section is that the results in
\intref{Section}{se:supportsets} up to \intref{Section}{se:support-and-cosupport} carry over to the other elementary supergroup schemes recalled in \intref{Section}{se:prelim}. We begin with the supergroup scheme $\mcE$ with group algebra
\begin{equation}
\label{se:kE-exterior}
k\mcE \colonequals k[s_1,\dots,s_n,\sigma]/(s_1^p,\dots,s_n^p,\sigma^2)
\end{equation}
with $|s_i|=0$ and $|\sigma|=1$. As an algebra $k\mcE$ is the tensor product of the group algebra of an elementary abelian $p$-group of rank $n$, and an exterior algebra on one generator, $\sigma$. This is again a complete intersection so the results of \intref{Appendix}{se:appendix} apply. In particular, there is a notion of support sets for $k\mcE$-modules, which are subsets of $\bbP^{n}(k)$, and these detect projectivity. The theory of $\pi$-points also carries over, and \intref{Theorem}{th:detection} carries over to these algebras. The only difference is that the representatives of $\pi$-points are given by maps $\alpha_{\bsa}\colon \test_K\to K\mcE_K$, where $K$ is a field extension of $k$, and 
\begin{align}
\label{eq:alpha-exterior}
  \begin{split}
     \alpha_{\bsa}(t) &=  a_1 s_1 + \cdots +  a_n s_n\,, \\
      \alpha_{\bsa}(\tau)&=a_{n+1}\sigma \\
  \end{split}
\end{align}
for $\bsa\in \bbA^{n+1}(K)$;  cf.~ \eqref{eq:alpha}. This has the
consequence that the rank variety of a finitely generated
$k\mcE$-module $M$, defined only when $k$ is algebraically closed,
consists of those points $\bsa\in \bbA^{n+1}(k)$ for which the square zero map
\begin{equation}
\label{eq:Mmatrix-exterior}
  \begin{pmatrix} 
  \alpha_{\bsa}(\tau) & \alpha_{\bsa}(t) \\ 
  -\alpha_{\bsa}(t) &    -\alpha_{\bsa}(\tau) 
  \end{pmatrix} \colon M\oplus M \lra M\oplus M
 \end{equation} 
 has rank at most $\rank_kM$. It is an affine cone in $\bbA^n(k)$;
 what is more, it is invariant under the $\bbZ/2$-action that negates
 the last coordinate: $a_{n+1}\mapsto -a_{n+1}$, because $M$ is
 $\bbZ/2$-graded. The map $V^r_{\mcE}(M)\to \hysupp_{\mcE}(M)$ that
 comes up in the proof of \intref{Theorem}{th:carlson} is equivariant
 with respect to the usual Frobenius map.

Finally we come to elementary abelian $p$-groups. The observation that tackles this case applies more generally to any finite  group scheme $\mcE$, viewed as a supergroup scheme. For such an $\mcE$ and any map of $k$-algebras $\alpha\colon \test\to k\mcE$, one has  $\alpha(\tau)=0$ for degree reasons, so there is an induced map of $k$-algebras
\[
\bar{\alpha} \colon \test/(\tau) \lra k\mcE\,.
\]
Observe that $\test/(\tau)\cong k[t]/(t^p)$, so $\bar{\alpha}$ is
a candidate for a  $\pi$-point of $\mcE$, in the sense of \cite{Friedlander/Pevtsova:2007a}.

\begin{lemma}
Let $\mcE$ be any  finite group scheme.  A map of $k$-algebras $\alpha\colon \test \to k\mcE$ has finite flat dimension if and only if the induced map $\bar{\alpha}\colon \test/(\tau) \to k\mcE$ is flat.
\end{lemma}

\begin{proof}
Since $\tau$ is not a zero divisor in $\test$, the canonical surjection  $\ve\colon \test\to \test/(t)$ has flat dimension one. Thus if $\bar{\alpha}$ is flat, $\alpha$ has finite flat dimension.

Assume the flat dimension of $\alpha$ is finite, that is to say, the
flat dimension of $k\mcE$ viewed as a module over $\test$ via $\alpha$
is finite. The action of $\test$ on $k\mcE$ factors through
$\test/(\tau)\cong k[t]/(t^p)$, and as a $k[t]/(t^p)$-module $k\mcE$
is a direct sum of copies of the cyclic modules $k[t]/(t^i)$, for
$1\le i\le p$. Evidently as an $\test$-module $k[t]/(t^i)$ has finite
flat dimension only when $i=p$; see \eqref{eq:resol}. Thus the
hypothesis on $\alpha$ implies that $k\mcE$ is a direct sum of copies
of $k[t]/(t^p)$, and hence  $\bar{\alpha}$ is flat, as claimed.
\end{proof}

Given the preceding result it is clear that $\pi$-points for elementary abelian $p$-groups defined via $\test$ have the same properties as the  ones introduced by Carlson~\cite{Carlson:1983a}, so nothing more needs to be said about this situation.

\appendix
\section{Complete intersections}
\label{se:appendix}
In this section we describe support sets of  modules over complete intersections, following Avramov~\cite{Avramov:1989a}, with a view towards detecting finite flat dimension of modules over complete intersections. With this in mind  we have chosen to work with complete intersection rings that are quotients of power series rings, and modules of finite Loewy length. This simplifies the exposition at various points. This utilitarian approach is also why we do not develop the material beyond \intref{Theorem}{th:appendix}, which can be taken as a starting point for a $\pi$-point approach to arbitrary modules over complete intersections.

Throughout this appendix $k$ will be field; there are no restrictions on the characteristic. Set $P\colonequals \pos {\bsx}$, the power ring over $k$ in indeterminates $\bsx\colonequals x_1,\dots,x_c$. In particular $P$ is a noetherian local ring, with maximal ideal $(\bsx)$.  It is to ensure these properties that we work with the ring of formal power series. An alternative would be to take the localisation of the polynomial ring over $(\bsx)$, localised at  $(\bsx)$.
 
Let $\bsf\colonequals f_1,\dots,f_c$ be elements in $P$ such that 
\begin{enumerate}
\item
$(\bsf)$ is in $(\bsx)^2$, and 
\item 
$f_i$ is not a zero-divisor in $P/(f_1,\dots, f_{i-1})$ for $1\le i\le i$.
\end{enumerate}
Condition (2) states that the $\bsf$ is a regular sequence in $P$, and hence the ring
\[
R\colonequals P/(f_1,\dots,f_c)
\]
is a complete intersection local ring, of codimension $c$ in $P$. Let $\fm$ denote the maximal ideal $(\bsx)R$ of $R$. 

\subsection*{Generic hypersurfaces}
For each $\bsa \colonequals [a_1,\dots,a_{c}]$ in $\bbP^{c-1}(k)$ set
\[
h_{\bsa} (\bsx)\colonequals a_1f_1 + \cdots + a_cf_c \quad \text{and}\quad P_{\bsa}\colonequals P/(h_{\bsa}(\bsx))\,.
\]
While $h_{\bsa}(\bsx)$ depends on a representative for $\bsa$, the ideal it generates does not, so there is no ambiguity in the notation $P_{\bsa}$. Since $h_{\bsa}(\bsx)$ is in $(\bsf)$ there is a surjection 
\[
\beta_{\bsa}\colon P_{\bsa}  \lra R\,.
\]
Let $M$ be an $R$-module and $\beta^*_{\bsa}(M)$ its restriction to $P_{\bsa}$.  We shall be interested in the \emph{support set} of $M$, defined to be:
\begin{equation}
\hysupp_R(M)\colonequals \{\bsa\in \bbP^{c-1}(k) \mid \fdim \beta_{\bsa}^*(M) =\infty \}
\end{equation}
The language is borrowed from \cite[\S3]{Avramov/Iyengar:2018a},
whilst the result below is from~\cite[Corollaries 3.11,
3.12]{Avramov:1989a}. In these sources, projective dimension, rather
than flat dimension, is used, but this makes no difference, for one is finite
if and only if the other is finite, and the dimensions coincide when $M$ is finitely
generated, because finitely generated flat modules are projective.

\begin{theorem}
\label{th:Av}
Assume $k$ is algebraically closed and that $M$ is a finitely generated $R$-module. The subset $\hysupp_{R}(M)\subseteq \bbP^{c-1}(k)$ is closed in the Zariski topology. One has $\hysupp_{R}(M)=\varnothing$ if and only if $M$ has finite flat dimension. \qed
\end{theorem}

We need an extension of the preceding result that applies also to infinitely generated modules. For simplicity, we consider only $R$-modules $M$ of \emph{finite Loewy length}, meaning $\fm^s M =0$ for some $s\gg 0$. Such modules form a Serre subcategory of $\Mod R$, though not a localising subcategory. The main case of interest is when $R$ has codimension $n$, equivalently, of  Krull dimension zero,  in which case each $R$-module  has this property. However the proof of  \intref{Theorem}{th:appendix} below involves  an induction on the codimension of $R$, and then it becomes convenient to work with modules of finite Loewy length.

The result however requires only that the module is $\fm$-torsion, that is to say that each element is annihilated by a power of $\fm$.

\begin{lemma}
\label{le:pd-test}
When $M$ is  $\fm$-torsion the following conditions are equivalent:
\begin{enumerate}[\quad\rm(a)]
\item
$\fdim_RM < \infty$
\item
$\Tor^R_i(k,M)=0$ for $i\gg 0$;
\item
$\Ext_R^i(k,M)=0$ for $i\gg 0$.
\end{enumerate}
\end{lemma}

\begin{proof}
Any complete intersection ring is Gorenstein, and when such a ring has finite Krull dimension---for example, if it is a quotient of $P$---a module has finite flat dimension if and only if it has finite injective dimension; see \cite[(3.3.4)]{Christensen:2000a}.  Given this observation and the  hypothesis that $M$ is $\fm$-torsion, the equivalence of the stated conditions follows from \cite[Propositions 5.3.F, 5.3.I]{Avramov/Foxby:1991a}.
\end{proof}

For what follows we need to consider support sets defined over
extension fields.

\begin{definition}\label{de:extending-fields}
Let $K$ be an extension field of $k$. We set $P_K\colonequals \pos[K]{\bsx}$ and consider $P$ as a subring of $P_K$ in the obvious way. For any quotient ring $A$ of $P$, set $A_K\colonequals P_K\otimes_PA$; this is then a quotient ring of $P_K$. It is easy to verify that the extension $P\to P_K$ is flat, and hence so is the extension $A\to A_K$. 

In particular, $R_K$ is the quotient of $P_K$ by the ideal generated
by $f_1,\dots,f_c$, viewed as elements of $P_K$.  Hence $R_K$ is a
complete intersection of codimension $c$ in $P_K$, with maximal ideal
$\fm R_K$.  For an $R$-module $M$ set
\[
M_K \colonequals R_K\otimes_R M \qquad\text{and}\qquad M^K\colonequals \Hom_{R}(R_K,M)
\]
viewed as $R_K$-modules.
\end{definition}

\begin{remark}\label{re:extending-fields}
  The modules $M_K$ and $M^K$ have finite Loewy length provided that
  $M$ has finite Loewy length. This is clear since $\fm^s M=0$ implies
  $\fm^s M_K=0 = \fm^s M^K$.
\end{remark}

\begin{lemma}
\label{le:ext-tor}
When $M$ is an $R$-module of finite Loewy length, for each $i$ one has 
\begin{align*}
\Tor^{R_K}_i(K,M_K) &\cong K\otimes_k \Tor^R_i(k,M) \\
\Ext_{R_K}^i(K,M^K) &\cong \Hom_k(K,\Ext_R^i(k,M))\,.
\end{align*}
\end{lemma}

\begin{proof}
We verify the isomorphism involving Ext; the first one is  a tad easier to verify for $R_K$ is flat over $R$. By hypothesis, there exists a positive integer $s$ such that $\fm^s M=0$. Set $A\colonequals R/\fm^s$; then $A_K \cong R_K/\fm^s R_K$. Observe that $A$ is finite dimensional over $k$; given this, it is easy to verify that  $A_K$ is projective as an $A$-module. This justifies the second isomorphism in $\dcat{R_K}$ below
\[
\RHom_R(R_K,M) \cong  \RHom_A(A_K,M)\cong \Hom_A(A_K,M)\,.
\]
The first one is adjunction. We conclude that $\Hom_R(R_K,M) \cong \RHom_R(R_K,M)$. This justifies the first isomorphism below
\begin{align*}
\Ext^i_{R_K}(K,M^K) 
	&\cong  \Ext^i_{R_K}(K,\RHom_R(R_K,M)) \\
	& \cong \Ext^i_R(K,M) \\
	& \cong \Hom_k(K,\Ext^i_R(k,M))\,.
\end{align*}
The second one is  standard adjunction and the last one is just standard.
\end{proof}

We also need the following remarks concerning cohomology operators; we focus on codimension two for this is all that is needed in the sequel.
See \cite{Avramov/Buchweitz:2000a} for details.

\begin{remark}
\label{re:codim2}
Let $P\to Q\to R$ be a factorisation of the  surjection  $P\to R$ such that the kernel of the map $Q\to R$ can be defined by a $Q$-regular sequence, say $g_1,g_2$.  Set $\fn \colonequals (\bsx) Q$; this is a maximal ideal of $Q$ lying over the maximal ideal $\fm$ of $R$. Set $J\colonequals (g_1,g_2)$. There is a natural embedding of $k$-vector spaces
\[
\Hom_k(J /{\fn J}, k) \hookrightarrow \Ext^2_R(k,k)\,.
\]
The residue classes of $g_1,g_2$ are a basis for the $k$-vector space $J/\fq J$; let $\chi_1,\chi_2$ denote the image in $\Ext^2_R(k,k)$ of the dual basis. These are the cohomology operators constructed by Gulliksen~\cite{Gulliksen:1974a} and Eisenbud~\cite{Eisenbud:1980a} associated to  $Q\to R$. They lie in the center of the $k$-algebra $\Ext_R(k,k)$. 

Fix a point $\bsb\colonequals (b_1,b_2)$ in $k^2\setminus \{0\}$ and
set $Q_{\bsb}\colonequals Q/(b_1g_1+b_2g_2)$. Then one has a
surjection $Q_{\bsb}\to R$, with kernel generated by
$J/(b_1g_1+b_2g_2)$; it is easy to check that this ideal can be generated  any element $c_1g_1+c_2g_2$ such that $(c_1,c_2)$ is not a scalar multiple of $(b_1,b_2)$; no such element is a zero divisor.  The surjection
\[
\frac J{\fn J} \twoheadrightarrow \frac J{(\fn J + b_1g_1+b_2g_2)}
\]
yields an embedding
\[
\Hom_k(\frac J{(\fn J + b_1g_1+b_2g_2)},k) \hookrightarrow \Hom_k(\frac J {\fn J}, k) \hookrightarrow \Ext^2_R(k,k)\,.
\]
 From this it follows that the cohomology operator corresponding to the surjection $Q_{\bsb}\to R$ is $b_2\chi_1 - b_1\chi_2$. 
 
For any $R$-module $M$, the standard change of rings spectral sequence associated to $Q_{\bsb}\to R$ yields an exact sequence 
\begin{equation}
\label{eq:chi-les}
\cdots \to \Ext^{i+1}_{Q_{\bsb}}(k,M) \to \Ext^{i}_{R}(k,M) \xra{b_2\chi_1 - b_1\chi_2} \Ext^{i+2}_{R}(k,M) \to \Ext^{i+2}_{Q_{\bsb}}(k,M)\to \cdots
\end{equation}
\end{remark}

 The result below extends part of \intref{Theorem}{th:Av}.
 
\begin{theorem}
\label{th:appendix}
Let $R$ be as above and $M$ an $R$-module of finite Loewy length. The following conditions are equivalent:
\begin{enumerate}[\quad\rm(1)]
\item
$\fdim_RM$ is finite;
\item
$\hysupp_{R_K}(M_K)=\varnothing$ for any field extension $K$ of $k$;
\item
$\hysupp_{R_K}(M^K)=\varnothing$ for any field extension $K$ of $k$.
\end{enumerate}
Moreover, in \emph{(2)} and \emph{(3)} it suffices to take for $K$ an algebraically closed extension of $k$
of transcendence degree at least $c-1$.
\end{theorem}

\begin{proof}
The proof combines ideas from \cite{Avramov:1989a} and \cite{Benson/Carlson/Rickard:1996a}.  \intref{Remark}{re:extending-fields} and \intref{Remark}{re:codim2} will be used implicitly in what follows.

\medskip

(1)$\Rightarrow$(2) For any $\bsa$ in $\bbP^{c-1}(k)$ the map $P_{\bsa}\to R$ is a complete intersection, of codimension $c-1$; in particular $\fdim_{P_{\bsa}}R$ is finite. Thus $\fdim_RM$ finite implies $\fdim_{P_{\bsa}}M$ is finite as well, and hence $\hysupp_R(M)=\varnothing$. It remains to observe that if $\fdim_RM$ is finite, then so is $\fdim_{R_K}M_K$; for example, by combining \intref{Lemma}{le:pd-test} and \intref{Lemma}{le:ext-tor}.

\medskip

(1)$\Rightarrow$(3)  This can be proved akin to the previous implication.

\medskip

(3)$\Rightarrow$(1) 
We argue by induction on $c$. The result is a tautology when $c=1$, for then $\hysupp_R(M)=\varnothing$ is equivalent to $\fdim_RM<\infty$, by definition of the support set. We can thus assume $c\ge 2$, and that the desired conclusion holds for complete intersections of codimension $c-1$.   Set $Q\colonequals P/(f_3,\dots,f_{c})$, so that one has a surjection $Q\to R$.

Since $R=Q/(f_1,f_2)$ and $f_1,f_2$ is a $Q$-regular sequence,  corresponding to the surjection $Q\to R$ one has a subspace $k\chi_1 + k\chi_2$ of $\Ext^2_R(k,k)$; see \intref{Remark}{re:codim2}.  Fix a field extension $K$ of $k$, and elements $b_1,b_2$ in $K^2$, not both zero. Set $Q_{\bsb}\colonequals Q_K/(b_1f_{1}+b_2f_{2})$. The map $Q_K\to R_K$ factors as $Q_K\to Q_{\bsb} \to R_K$.

We wish to apply the induction hypothesis to $Q_{\bsb}$. This ring is the quotient of $P_K$ by the ideal generated by $b_1f_1+b_2f_2,f_3,\dots, f_c$, and the latter is a regular sequence in $P_K$.  Moreover for any field extension $L$ of $K$, any $L$-linear combination of this sequence is an $L$-linear combination of $f_1,\dots,f_c$. Keeping in mind that $\Hom_{R_K}(R_L,M^K)\cong M^L$ as $R_L$-modules, it is now a routine verification that 
\[
\hysupp_{(Q_{\bsb})_L}(M^L)= \varnothing
\]
as a subset of $L^{c-1}$. The upshot is that $\fdim_{Q_{\bsb}}(M^K)$ is finite, by the induction hypothesis. Since the Krull dimension of $Q_{\bsb}$ is at most $n$, it follows that 
\[
\fdim_{Q_{\bsb}}M^K\le n\,.
\]
Fix $i\ge n$ and consider the following snippet 
\[
\Ext^{i+1}_{Q_{\bsb}}(K,M^K) \to \Ext^{i}_{R_K}(K,M^K) \xra{\ b_2\chi_1 -  b_1\chi_2\ } \Ext^{i+2}_{R_K}(K,M^K) \to \Ext^{i+2}_{Q_{\bsb}}(K,M^K)
\]
of the exact sequence from \eqref{eq:chi-les} associated to $Q_{\bsb}\to R_K$. Thus the choice of $i$, the exact sequence above, and  \intref{Lemma}{le:ext-tor} imply that one has an isomorphism
\[
\Hom_k(K, \Ext^{i}_R(k,M)) \xra[\cong]{\ b_2\chi_1 - b_1\chi_2\ } \Hom_k(K, \Ext^{i+2}_R(k,M))\,.
\]
Since $(b_1,b_2)$ in $K^2\setminus \{0\}$ was arbitrary,  \cite[Lemma~5.1]{\bikp:2017a} implies $\Ext^{i}_R(k,M)=0$. Since this holds for each $i\ge n$ we deduce that $\fdim_RM\le i$, by \intref{Lemma}{le:pd-test}.

\medskip

(2)$\Rightarrow$(1) can be proved along the lines of (3)$\Rightarrow$(1). Instead of  \cite[Lemma~5.1]{Benson/Iyengar/Krause/Pevtsova:2017a}, one applies its analogue for tensor products \cite[Theorem~5.2]{Benson/Carlson/Rickard:1996a}. These results also make it clear that it suffices to consider algebraic closed extension fields $K$ of transcendence degree $c-1$.
\end{proof}

\end{document}